\documentclass[twoside,11pt, letter]{article}

\usepackage{fixltx2e} 

\usepackage{cmap}

\usepackage[T1]{fontenc}
\usepackage[utf8]{inputenc}

\usepackage{verbatim}

\usepackage{booktabs}       

\usepackage{nicefrac}       

\usepackage{microtype}      

\usepackage{graphicx} 

\usepackage[colorlinks]{hyperref}
\usepackage{url}

\usepackage{geometry} 
\usepackage[toc,page]{appendix}
\usepackage{mathrsfs}
\usepackage{amsmath, amssymb, amsthm}
\usepackage[linesnumbered, ruled]{algorithm2e}
\usepackage{tikz}
\usepackage{float}
\usepackage{bm}
\usepackage{afterpage}

\usepackage[font=small,labelfont=bf,labelsep=period]{caption}

\usepackage{siunitx,array}
\newcolumntype{L}{>{\phantom{$-$}}l}       
\newcommand\mcL[1]{\multicolumn{1}{L}{#1}} %

\DeclareFontFamily{U}{matha}{\hyphenchar\font45}
\DeclareFontShape{U}{matha}{m}{n}{
      <5> <6> <7> <8> <9> <10> gen * matha
      <10.95> matha10 <12> <14.4> <17.28> <20.74> <24.88> matha12
      }{}
\DeclareSymbolFont{matha}{U}{matha}{m}{n}
\DeclareMathSymbol{\odiv}         {2}{matha}{"63}

\newcommand{\asto}{\longrightarrow} 
\newcommand{\probto}{\longrightarrow} 
\newcommand{\lawto}{\longrightarrow} 

\newcommand{\prob}{\ensuremath{\mathbb{P}}} %
\newcommand{\expect}{\ensuremath{\mathbb{E}}} %

\newcommand{\graphset}{\mathcal{G}}
\newcommand{\vertexset}{\mathcal{V}}
\newcommand{\edgeset}{\mathcal{E}}
\newcommand{\neighbors}{\mathcal{N}}
\newcommand{\infected}{\mathcal{I}}

\newcommand{\vect}[1]{\bm{#1}}
\newcommand{\diag}{\text{diag}}

\newcommand{\textout}{\text{out}}

\newcommand{\history}{\field}
\newcommand{\hessian}{H}
\newcommand{\rank}{\ensuremath{\operatorname{rank}}}
\newcommand{\cov}{\ensuremath{\operatorname{cov}}}
\newcommand{\field}{\mathcal{F}}
\newcommand{\normal}{\mathscr{N}}
\newcommand{\unordered}{\mathcal{U}}
\newcommand{\ordered}{\mathcal{O}}
\newcommand{\proj}{\text{proj}}
\newcommand{\E}{\expect} 
\newcommand{\real}{\ensuremath{\mathbb{R}}}
\newcommand{\var}{\ensuremath{\operatorname{var}}} 
\newcommand{\defeq}{=}
\newcommand{\trans}{T}

\newcommand{\vbeta}{\vect{\beta}}
\newcommand{\sbeta}{\beta}

\newcommand{\reals}{\mathbb{R}}
\newcommand{\ra}[1]{\renewcommand{\arraystretch}{#1}}
\newcommand{\intext}{\text{ in }}

\newcommand{\citep}[1]{\cite{#1}}
\newcommand{\citet}[1]{\cite{#1}}
\newcommand{\citealp}[1]{\cite{#1}}

\newtheorem{theorem}{Theorem}
\newtheorem{lemma}[theorem]{Lemma}
\newtheorem{proposition}[theorem]{Proposition}
\newtheorem{corollary}[theorem]{Corollary}
\newtheorem{definition}[theorem]{Definition}

\newtheorem{remark}[theorem]{Remark}

\newtheorem{theorem*}[theorem]{Theorem}   
\newtheorem{lemma*}{Lemma} 
\newtheorem{corollary*}{Corollary} 
\newtheorem{remark*}{Remark}
\newtheorem{example*}{Example}
\newtheorem{definition*}{Definition}
\newtheorem{proposition*}{Proposition}

\newfloat{newtable}{thp}{lop}
\floatname{newtable}{Table}

\title{A theory of maximum likelihood for weighted infection graphs}

\author{Justin Khim\thanks{Department of Statistics, University of Pennsylvania, Philadelphia, PA 19104.}
\and 
Po-Ling Loh\thanks{Department of Electrical \& Computer Engineering and Statistics,
       University of Wisconsin,
       Madison, WI 53706.}}

\begin{document}

\maketitle

\begin{abstract}
We study the problem of parameter estimation based on infection data from an epidemic outbreak on a graph. We assume that successive infections occur via contagion; i.e., transmissions can only spread across existing directed edges in the graph. Our stochastic spreading model allows individual nodes to be infected more than once, and the probability of the transmission spreading across a particular edge is proportional to both the cumulative number of times the source nodes has been infected in previous stages of the epidemic and the weight parameter of the edge. We propose a maximum likelihood estimator for inferring the unknown edge weights when full information is available concerning the order and identity of successive edge transmissions. When the weights take a particular form as exponential functions of a linear combination of known edge covariates, we show that maximum likelihood estimation amounts to optimizing a convex function, and produces a solution that is both consistent and asymptotically normal. Our proofs are based on martingale convergence theorems and the theory of weighted P\'{o}lya urns. We also show how our theory may be generalized to settings where the weights are not exponential. Finally, we analyze the case where the available infection data comes in the form of an unordered set of edge transmissions. We propose two algorithms for weight parameter estimation in this setting and derive corresponding theoretical guarantees. Our methods are validated using both synthetic data and real-world data from the Ebola spread in West Africa.
\end{abstract}


\section{Introduction}
\label{secIntro}

Information, behavior, and diseases often spread over the edges of an underlying network~\citep{AndMay92, aral2011, christakis2007, Jac08}. On Twitter or Facebook, users may share information with their friends and followers~\citep{KatEtal11, LesEtal07}; in communities or countries, individuals infected with HIV or Ebola may spread the disease to other people and other regions by direct physical contact~\citep{HelKoh07, Kis14, NeaEtal94, RicEtal15}. As more data become available and researchers are able to trace the connections between individuals and recover information about specific transmission events, several novel statistical questions have emerged concerning estimation of unknown parameters governing the stochastic spread. Some scientific questions of interest include the following: Which individuals are at the greatest risk of being infected by HIV in a relationship network? What treatments are effective at slowing the spread of disease from one region or person to another~\citep{drakopoulos2016, hoffmann2017}? Applying the same probabilistic models to an entirely different setting may allow researchers in online marketing and advertising to devise an optimal advertising budget for product proliferation in online social networks~\citep{DomRic01, kempe2003, Sco15}. Although a sizable amount of literature exists for answering these questions based on a particular network and stochastic spreading model~\citep{CheEtal09, SahEtal15, SomEtal14}, relatively little existing work addresses the problem of estimating a particular stochastic model based on infection transmission data.

In the statistics literature, most research in network science has focused on theory and applications of graphical model estimation~\citep{Lau96, WaiJor08}. However, the equally important goal of performing further estimation or inference procedures based on a known (or estimated) graph structure has been largely unaddressed, with an exception being graph partitioning approaches such as stochastic block modeling~\citep{Abb17, WolOlh13, ZhaZho16}. Another novel aspect of our work is that we do not assume our data are collected in an independent, identically distributed manner, as might be the case if one were observing multiple diseases spread on the same graph. In contrast, our inference procedures are based on information collected about a single epidemic outbreak, which is more realistic in various applications. Indeed, although different diseases may spread over the same network, the spreading behavior of each disease may be quite different, or the disease may exhibit different characteristics during successive spreads; hence, the goal is simply to model the propagation of the disease under consideration. Notably, although epidemiology has traditionally been an active application area in statistics~\citep{Jew03, KypMin18, RotEtal08}, relatively few methods exist in previous literature that take into account network structure within a population of interest.

Outside of statistics, several attempts have been made to estimate stochastic models for the spread of information and disease, but they are far from comprehensive. In one line of work, each vertex may only be infected once, and time information providing the order in which individuals were infected is unavailable. This model, applicable to settings such as HIV, is considered in \cite{milling2015}, as well as in our own previous work~\citep{khim2017}. Another variation relies on observing multiple independent, identically distributed cascades, where a cascade is an infection in which each vertex may be infected only once~\citep{gomez2016, netrapalli2012}. While the notion of cascades was likely inspired by applications in social networks, obtaining accurate data from actual cascades is often much more complicated~\citep{myers2012}. Finally, Bayesian approaches that leverage virus DNA sequences have been used in the epidemiology literature \citep{dudas2017, lemey2014}. However, these techniques require manually segmenting time into epochs. Furthermore, none of the methods described above for parameter estimation come with statistical guarantees concerning inference.

In this paper, we consider a parametrized model of transmission inspired by percolation~\citep{kesten1993} and gravity models. Our goal is to perform statistical inference for the unknown parameter vector that determines the edge weights which govern the spread of the infection. Furthermore, we associate each edge of the graph to a vector of known covariates. This allows us to answer questions such as the following: In the week leading up to a sports match, writers and fans post on Twitter. Are friend-follower relationships important for the order in which posts occur? Do previous interactions predict future interactions? As a second example, suppose we wish to reconstruct the spread of Ebola in West Africa. Does the reconstruction reveal which factors, such as proximity or language, best correlate with the spread of the disease? This paper provides methods to analyze such datasets. Our analysis crucially leverages theoretical properties of P\'{o}lya urn processes~\citep{pemantle2007}: Specifically, the counts and weights of balls of a given color in the urn are separated, and the weights are to be inferred based on observing ball counts. To the best of our knowledge, this decoupling between counts and weights has not been previously studied in urn theory and provides a modeling tool of independent interest. We then allow for two scenarios: one in which we obtain the order of vertex infections, and another in which we observe an unordered set containing infection information.

Our main contributions are to analyze the maximum likelihood estimator and study its asymptotic properties. This allows us to construct confidence intervals and test the importance of covariates. Furthermore, we show that maximum likelihood estimation amounts to optimizing a convex objective function, which is therefore computationally tractable. The case where order information about successive transmissions is unavailable is far more difficult, since the maximum likelihood estimator becomes intractable to compute. In this scenario, we derive two equations that lead to algorithms for computing parameter estimates. Additionally, we provide a class of examples for which the two estimators coincide. Finally, we explore the empirical performance of our algorithms using synthetic spreading data and historical data obtained from the Ebola spread.

The paper is organized as follows: In Section~\ref{secModel}, we define our infection model.
In Section~\ref{secMLE},   we introduce and state the properties of the maximum likelihood estimate. Section~\ref{secExtensions} offers two extensions of the maximum likelihood estimation approach. Section~\ref{secWithoutOrder} contains estimators for the case of unordered infection data. In Sections~\ref{secSims} and \ref{secApplications}, we apply our methods to synthetic and real data, respectively. Finally, we conclude with a discussion of further research directions in Section~\ref{secDiscuss}. Due to space constraints, we provide proofs of all our theoretical results and tables of numerical results in the Appendix.


\section{Background and problem setup}
\label{secModel}

\subsection{Notation}

We begin by defining some basic notation. For a vector $\vect{y} \intext \real^m$, we write $\diag(\vect{y})$ to denote the $m \times m$ diagonal matrix with diagonal entries equal to $\vect{y}$. For two vectors $\vect{x}$ and $\vect{y}$, we use \(\vect{x} \circ \vect{y}\) to denote entrywise multiplication and $\vect{x} \odiv \vect{y}$ to denote entrywise division. 

\subsection{Stochastic spreading model}

Let $\graphset = (\vertexset, \edgeset)$ denote a directed graph, with vertex set \(\vertexset = \{1, \ldots, n\}\) and edge set $\edgeset \subseteq \vertexset \times \vertexset$, where we allow for self-loops. Let $m = |\edgeset|$. We also often identify the edges of $\graphset$ with the enumeration $\{1, \dots, m\}$. Each edge \((u, v) \intext \edgeset\) is assigned a covariate vector 
\(\vect{x}_{u, v} \intext \reals^{d}\) and a nonnegative weight \(w(u, v)\), possibly equal to a function of $\vect{x}_{u,v}$. We require that our graphs be {\em strongly connected}, meaning that between any two vertices, there is a path across edges of nonzero weight in each direction. We collect the covariate vectors as rows in a covariate matrix \(X \intext \real^{m \times d}\). Let $\neighbors_{\textout}(u) = \{v: (u,v) \text{ is in } \edgeset\}$ denote the set of out-neighbors of \(u\).

Next, we define the infection process spreading over the edges of $\graphset$. At time \(t = 1\), the infection originates from a single randomly chosen vertex \(v_{1}\). At each subsequent time \(t\), an infected vertex \(u_{t}\) is chosen among the set of previously infected vertices, and infects another vertex \(v_{t}\) along the edge $e_{t} = (u_t, v_t)$. We denote the \(\sigma\)-field of random infection events up to time \(k\) by \(\field_{k} = \sigma(v_{1}, e_{2}, \ldots, e_{k})\), and we denote the vector of ordered infection data by \(\ordered_{k} = (v_{1}, e_{2}, \ldots, e_{k})\). In the case of unordered infection data, we denote the set of unordered infection spreads by \(\unordered_{k} = \{e_{2}, \ldots, e_{k}\}\). We will also briefly discuss the case where only the ordered set of infected nodes \(\infected_{k} = (v_{1}, \ldots, v_{k})\) is observed, without information about which nodes $(u_2 , \dots, u_k)$ were the sources of transmission.

Importantly, a given node may be infected any number of times. Drawing an analogy to the scenario of epidemic spreading, we may treat each node as a different community, in which the infection count of a certain community may increase over time as more individuals become infected with the same disease. Transmissions between individuals in the same community correspond to self-loops in the edge infection data.

We now define the specific measure for our infection process that will be studied in this paper. Let $b_{t}(u) = |\{s < t: v_s = u\}|$ denote the number of times vertex \(u\) has been infected prior to time \(t\), and set \(b_{t}(e) = b_{t}(u)\) for \(e = (u, v)\). We define the infection probabilities by
\begin{align}
& \begin{aligned}
\prob\left(\ordered_{k}\right)
&=
\frac{1}{n} 
\prod_{t = 2}^{k}
\frac{b_{t}(e_{t}) w(e_{t})}{\sum_{e \in \edgeset} b_{t}(e) w(e)}.
\label{eqnModel}
   \end{aligned}
\end{align}
In other words, the first infected vertex is chosen uniformly at random. Each subsequent vertex \(v_{i}\) is chosen to be infected with probability proportional to the weights entering \(v_{i}\) from all previously infected vertices. This is equivalent to a continuous-time infection process where each new infection instantiates an independent exponential random variable along all outgoing edges with parameter proportional to the corresponding edge weight, corresponding to the waiting time until its neighbor becomes infected.

We will focus on the specific case where the weight function takes the form
\begin{align}
& \begin{aligned}
w_{\vbeta}(i, j) 
&=
\exp\left(\vect{x}_{i,j}^T \vect{\beta}\right),
\label{eqnExpWeight}
   \end{aligned}
\end{align}
where \(\vect{\beta}\) is a parameter vector in \(\reals^{d}\).
This formulation resembles the exponential form of infections given by \cite{dudas2017} and is useful for interpretability of covariates affecting a spreading process. Furthermore, as explained in Remark~\ref{RemGeneralWeights} below, the case of more general weight functions may be recast in terms of the exponential weight formulation~\eqref{eqnExpWeight} when the weights all lie in the interval \([0, 1].\) For more details, see Section~\ref{SecGeneral}.

Finally, we define the notation we will use for conditional and limiting measures.
We write \(\prob_{\vbeta}\) to denote the measure defined by equations~\eqref{eqnModel} and~\eqref{eqnExpWeight}, where \(\vbeta\) is the underlying parameter. We will generally use \(\vbeta_{0}\) to denote the true parameter. We define \(\prob_{\vbeta_{0}, t}((u,v))\) as shorthand for the conditional measure  \(\prob_{\vbeta_{0}}\left(e_t = (u,v) | \history_{t - 1}\right)\), meaning that the identity of the next infected edge \(e_{t}\) is drawn conditional on the infection history \(\history_{t - 1}\). We will often use the vector $Z_t$ to denote the vector of edge covariates corresponding to the edge $e_t$ drawn according to $\prob_{\vbeta_0, t}$. In detail, the probability space of the entire infection process is given by \((\Omega, \mathcal{F}, \prob_{\vbeta_{0}})\), so \(\prob_{\vbeta_{0}, t}\) is a random measure; since random variables are functions of their sample space, \(\prob_{\vbeta_{0}, t}(\cdot)\) is defined on \(\Omega \times \edgeset\).

We also define the limiting measure \(\prob_{\vbeta_{0}, \infty} = \lim_{t \to \infty} \prob_{\vbeta_{0}, t}\). A priori, it is not clear that such a limit must exist, but we will subsequently derive the existence of the limit as a consequence of our P\'{o}lya urn theory. We will also refer to the limiting measure, which turns out to be non-random, as \(\vect{\pi}_{\vbeta_0}\), when considered as a vector.


\subsection{Generalized P\'{o}lya urns}
\label{subsecPolya}

In this section, we briefly review the theory of generalized P\'{o}lya urns, which is an important component of our theoretical analysis~\citep{pemantle2007}. In a generalized P\'{o}lya urn, we have an urn containing balls of \(m\) different colors. In addition, each ball is labeled with a positive real-valued weight. The evolution of the urn is governed by an \(m \times m\) replacement matrix \(W\). At each time \(t\), a ball is drawn from the urn with probability proportional to its weight. If a ball of color \(i\) is drawn, it is returned to the urn along with one ball of color $j$ and weight $W_{ij}$, for each $j$ such that $W_{ij} > 0$.

To relate this urn model to the spreading process described by equation~\eqref{eqnModel}, consider a generalized P\'{o}lya urn defined by the edges of the infection graph. The colors of the balls in the urn correspond to the $m$ edges. We now describe the replacement matrix \(W\): The rows and columns of $W$ are indexed by the edges of the graph. For each pair of edges $(e,f)$, where $e = (u,v)$ and $f = (x,y)$, we define $W_{ef}$ to be nonzero exactly when $v = x$, in which case $W_{ef} = w(e)$. It is not hard to see that successive infections on the graph follow the same probabilistic mechanism as the urn evolution, where \(b_{t}(e)\) corresponds to the number of balls for edge \(e\) present in the urn at time \(t\). Thus, the goal is to infer the weight matrix of the P\'{o}lya urn based on ball counts. Note that if the first infected vertex is \(v_{1}\), we can consider the urn as being initialized with one ball of weight $w(e)$ for each edge \(e = (v_{1}, u)\), where $u$ is in $\neighbors_{\textout}(v_1)$.


\subsection{Relation to gravity models}

Finally, we comment on the relationship between our proposed model and the problem setup adopted in the literature on gravity models.
The term ``gravity model'' is used broadly to describe a model for a quantity that depends proportionally on certain quantities and inversely on certain others. The terminology is inspired by Newton's law of universal gravitation, in which the gravitational force between two point masses is proportional to the product of the two masses and the inverse of the square of the distance between them.

In social science, a seminal work on gravity models observed that the number of graduates of a university depends proportionally on the population of the region and inversely on the distance from the university~\citep{stewart1941}. The book by \cite{sen1995} considers gravity models in great detail. More recent works such as \cite{viboud2006} have developed gravity models for susceptible, infected, recovered (SIR) models of network infection. The underlying random variables are generally modeled as binomial, and the estimation procedure consists of fitting parameters that control the spread of disease from one region to another.

Other recent attempts at disease modeling~\citep{dudas2017, lemey2014} have utilized a continuous-time Markov process governed by an \(n \times n\) infinitesimal rate matrix \(\Lambda\), where
\begin{equation}
\log \Lambda_{(u, v)}
=
\sbeta_{1} x_{(u, v), 1} + \cdots + \sbeta_{d} x_{(u, v), d}.
\label{eqnGravityLambda}
\end{equation}
Similar to our setup, any subsequent infection occurs over the edge \((u, v)\) with probability proportional to some \(\Lambda_{(u, v)} = \exp(\vect{x}_{(u, v)}^{T} \vbeta)\). In addition, the number of infections over \((u, v)\) under equation~\eqref{eqnGravityLambda} over a pre-specified period of time, such as a day, is a Poisson random variable with mean  \(\Lambda_{(u, v)}\), as is the case for the standard gravity model of \cite{sen1995}.

However, a key difference between the model adopted by this approach and the setting of our paper is the reinforcing aspect of our process. In our model, the function \(b_{t}: \vertexset \rightarrow \real\) increases monotonically in $t$, meaning that the amount of infection at each node accumulates over time. In contrast, the gravity model of \cite{sen1995} does not contain this feature, and is better suited for applications such as modeling road traffic between two cities, which might reasonably be constant when adjusted for seasonality. Similarly, the spread of influenza considered in \cite{lemey2014} might fit well into this gravity model framework, partly because influenza is quite ubiquitous and the infection process is somewhat stationary.

However, cases may exist where these assumptions are unreasonable. In particular, the historical spread of Ebola was observed to pass through various phases where, at first, it was not widespread enough to travel over most edges~\citep{dudas2017}. In order to facilitate the methods described above, the Ebola spread was manually split into three phases, and separate parameters were fit for each phase. However, the information on where to partition the infection data was determined by first inspecting the data, leading to unaddressed questions involving complicated dependences and post-selection inference.
The reinforcement aspect built into the P\'{o}lya urn formulation obviates the need to manually segment data. Another important distinction between our work and the literature on gravity models is that we provide rigorous statistical theory for the estimation algorithms we propose. Lastly, note that we consider a discrete-time model of infection spread rather than a continuous-time model, out of convenience; hence, one downside of our model is that we are not able to infer how the rate of spreading might evolve in continuous time.


\section{Maximum likelihood estimation}
\label{secMLE}

\subsection{Log-likelihood expression}

We now investigate the problem of maximum likelihood estimation for the model defined above.
Let \(L(\vbeta; \ordered_{k}) = \prob_{\vbeta}(\ordered_{k})\) denote the likelihood of \(\ordered_{k}\), computed with respect to the parameter vector \(\vbeta\). The log-likelihood for the exponential parametrization is then given by
\begin{align}
& \begin{aligned}
\ell(\vbeta; \ordered_{k}) 
&=
\sum_{t = 2}^{k} 
\left(\vect{x}_{e_{t}}^{T} \vbeta
+ \log b_{t}(e_{t}) \right)
- 
\sum_{t = 2}^{k} 
\log\left[
\sum_{e = 1}^{m}  
b_{t}(e) \exp\left(\vect{x}_{e}^{T} \vbeta \right) \right]
- \log n.
\label{eqnExpLikelihood}
   \end{aligned}
\end{align}
Importantly, note that the objective function~\eqref{eqnExpLikelihood} is concave in $\vbeta$, since it is a difference of linear terms and log-sum-exp functions~\citep{BoyVan04}. Hence, we may compute the maximum likelihood estimator efficiently via convex programming. More details are contained in Section~\ref{secSims}.



In the succeeding sections, we will consider the existence, computability, and statistical properties of the maximum likelihood estimator.
Note that we denote the true parameter by \(\vect{\beta}_{0}\) and the maximum likelihood estimator by \(\hat{\vect{\beta}}_{k}\).

\subsection{Existence}
\label{subsecMLEExistence}

We first establish conditions for the existence of the maximum likelihood estimator. We make the following definition:

\begin{definition}
We say that the data $\ordered_k$ satisfy the {\em suboptimal sampling condition} if for each nonzero \(\vect{v} \in \reals^{d}\),
either
\begin{itemize}
\item[(i)] there exists a time $1 \le t \le k$ and an edge $f \in \edgeset$ such that \(b_{t}(f) > 0\) and 
\(\vect{x}_{e_{t}}^T \vect{v} < \vect{x}_{f}^T \vect{v}\); or
\item[(ii)] for all $1 \le t \le k$ and $f \in \edgeset$ such that \(b_{t}(f) > 0\), we have \(\vect{x}_{f}^T \vect{v} = c\) for some constant \(c\), independent of $t$ and $f$.
\end{itemize}
\end{definition}

Note that the condition $b_t(f) > 0$ simply means the source vertex of $f$ has accumulated at least one infection prior to time $t$. Furthermore, the condition $\vect{x}_{e_t}^T \vect{v} < \vect{x}_f^T \vect{v}$ is equivalent to the condition $w_{\vect{v}}(e_t) < w_{\vect{v}}(f)$; i.e., under parameter vector $\vect{v}$, the edge $f$ has heavier weight than $e_t$.

We then have the following theorem, which states that the suboptimal sampling condition is both necessary and sufficient to guarantee the existence of a maximizer:

\begin{theorem}
The suboptimal sampling condition is satisfied if and only if the likelihood function attains a maximum.
\label{propExistenceConditions}
\end{theorem}

The proof can be found in Appendix~\ref{appExistUnique}. The suboptimal sampling condition is inspired by a relatively simple idea: In order to compare the propensity of one edge to spread infection versus another, we need instances in which infections are possible across both edges. Furthermore, we need instances of both spreads happening, or else our estimator would attempt to assign a probability of \(1\) to spreading across a particular edge.

\begin{remark}
Currently, we do not have a simple method to check whether the suboptimal sampling condition holds by hand, but it is possible to check the condition by determining whether an optimization problem with an arbitrary objective and linear inequality constraints is feasible.
\end{remark}

We now return to the possible non-existence of the maximum likelihood estimator.
To make an instance of non-existence more precise,
we offer an easy corollary to Theorem~\ref{propExistenceConditions}:

\begin{corollary}
Suppose there is a
$\vect{\beta} \intext \real^d$ such that 
\begin{equation}
\label{EqnNoMLE}
\vect{x}_{v_{s}, v}^{T} \vect{\beta} < \vect{x}_{u_{t}, v_{t}}^{T} \vect{\beta},
\end{equation}
for all $s \le t$ and $v \intext \neighbors_{\textout}(v_s)$ such that \((v_{s}, v) \neq (u_t, v_t)\).
Then the maximum likelihood estimator does not exist. 
\label{corBetaStarNoMLE}
\end{corollary}

As a basic example in one dimension, if \(\vect{x}_{u_{t}, v_{t}}\) is always the largest of the \(\vect{x}_{i}\)'s with \(b_{t}(i) > 0\); i.e., all spreads happen on the edges with the heaviest weights, the condition~\eqref{EqnNoMLE} is satisfied with $\vect{\beta}$ equal to any vector with strictly positive entries. Corollary~\ref{corBetaStarNoMLE} also illustrates the meaning of the term ``suboptimal sampling condition," since it gives a case where only the heaviest-weight vertices receive a transmission event on any time step.


\subsection{Uniqueness}
\label{subsubsecUniqueness}

Our goal is to show that the log-likelihood is strictly concave.
We start by establishing a helpful representation for the Hessian of the log-likelihood:

\begin{proposition}
Let \(Z_{t}\) be sampled from the \(\vect{x}_{e}\)'s with probability proportional to \(b_{t}(e) \exp(\vect{x}_{e}^{T} \vbeta)\).
Let \(C_{t} \defeq \cov_t(Z_t, Z_t)\) denote the covariance matrix of \(Z_{t}\) under \(\prob\left(\cdot \mid \field_{t - 1}\right)\).
The Hessian of the log-likelihood is given by
\[
\hessian\big(\ell(\vbeta; \ordered_{k})\big)
=
-\sum_{t = 2}^{k} C_{t}.
\]
In particular, if $\rank(C_{t}) = d$ for some $t$, the log-likelihood is strictly concave.
\label{propLogLikelihoodHessian}
\end{proposition}

The proof proceeds via a straightfoward calculation and is provided in Appendix~\ref{appCovRanks}. The next proposition provides a sufficient condition for having full-rank covariance matrices.

\begin{proposition}
Suppose \(b_{t}(e) > 0\) for all $e \intext \edgeset$,
and suppose \(\rank(X) = d\).
If there is a vector \(\vect{v}\) satisfying 
\[
X \vect{v} = \vect{1},
\]
then \(\rank(C_{t}) = d - 1\).
Otherwise, we have \(\rank(C_{t}) = d\).
\label{propCovRanks}
\end{proposition}

The proof is somewhat involved computationally and is deferred to Appendix~\ref{appCovRanks}. In general, we expect that \(\rank(C_{t}) = d\), unless \(X\) is degenerate and has a solution to \(X \vect{v} = \vect{1}\).

Note that for large enough $t$, we will have $\min_{1 \le i \le m} b_t(i) > 0$, almost surely. Hence, Propositions~\ref{propLogLikelihoodHessian} and~\ref{propCovRanks} together imply that if $\rank(X) = d$ and $X\vect{v} \neq \vect{1}$ for all $\vect{v}$, the log-likelihood is strictly concave for sufficiently large $t$.



\subsection{Asymptotic theory}
\label{subsecMLEProps}

We now turn to our main statistical result, which is a theorem establishing asymptotic properties of the maximum likelihood estimator. 

\begin{theorem}
Suppose that \(\rank(X) = d\) and \(X\vect{v} \neq \vect{1}\) for any \(\vect{v}\).
If \(\graphset\) is strongly connected and \(\vect{\beta}_{0}\) is \(d\)-dimensional, 
then the maximum likelihood estimator exists for sufficiently large $k$. Furthermore, it is consistent, and we have the convergence in distribution
\[
k^{-1/2}
(\hat{\vect{\beta}}_{k} - \vect{\beta}_{0})
\lawto 
\normal_{d}(0, I^{-1}_{\infty}(\vect{\beta}_{0})),
\]
where the \((a, b)\) coordinate of \(I_{\infty}(\vbeta_{0})\) is 
\[
I_{\infty, ab}(\vbeta_0) 
\defeq
\expect_{\vbeta_{0}}\left[
\frac{\partial}{\partial \sbeta(a)} \log \prob_{\vbeta_{0}, \infty}\{Z_\infty\} 
\frac{\partial}{\partial \sbeta(b)} \log \prob_{\vbeta_{0}, \infty}\{Z_\infty\}
\right].
\]
\label{theoremAsNormal}
\end{theorem}

The proof of Theorem~\ref{theoremAsNormal} is contained in Appendix~\ref{SecAsymp}. Although the general approach of deriving asymptotic normality resembles the techniques used in the standard theory of $M$-estimation in a multivariate setting, several technical challenges arise due to the fact that our objective function is not simply a sum of independent, identically distributed terms.

\begin{remark}
Recall from Corollary~\ref{corBetaStarNoMLE} that the maximum likelihood estimator may not exist. However,
as \(k\) tends to infinity, Theorem~\ref{theoremAsNormal} guarantees that the maximum likelihood estimator exists, meaning that the suboptimal sampling condition mentioned in Theorem~\ref{propExistenceConditions} must eventually be satisfied for sufficiently large $k$. This agrees with the intuition that as $k$ increases, the probability of spreads happening exclusively on edges of largest weight decays to 0.
\end{remark}


\section{Extensions}
\label{secExtensions}

\subsection{General weight functions}
\label{SecGeneral}

We now briefly discuss extensions of the maximum likelihood theory we have derived for exponential weights to slightly more general frameworks. We begin by considering the maximum likelihood estimator for non-exponentially parametrized weight functions.
We propose a two-step algorithm, described in Algorithm~\ref{algGeneralWeights}, to compute a maximum likelihood estimator:

\begin{algorithm}[!h]
\SetKwInOut{Input}{Input}
\Input{Transmission set \(\ordered_{k}\).}

For each edge \((u, v)\), if $(u, v)$ is not in $\ordered_k$, set \(\hat{w}(u, v) = 0\).

For each remaining edge \((u, v)\), define $\hat{z}(u, v)$ to be the corresponding component of the optimization problem
\begin{align}
\label{eqnGeneralOpt}
\hat{z} \in \arg\max_{z \in \real^{|\ordered_k|}} 
\hspace{10pt} 
& \sum_{t = 2}^{k} z(e_{t}) 
-
\sum_{t = 2}^{k} 
\log\left[
\sum_{e = 1}^{m} b_{t}(e) \exp(z(e)) \right],
\end{align}

and set $\hat{w}(u, v) = \exp\big(\hat{z}(u,v)\big)$.

\caption{General weights maximum likelihood computation}
\label{algGeneralWeights}
\end{algorithm}


A simple argument shows that Algorithm~\ref{algGeneralWeights} computes a maximum likelihood estimator in the case of general weights, if it exists: 
If $(u,v)$ is not in  $\ordered_k$ and $w(u,v) \neq 0$, the value of the likelihood is not decreased by setting $w(u,v) = 0$. Now consider $(u,v)$ in $\ordered_k$. Clearly, the weight \(w(u, v)\) must be positive in order for the likelihood to be nonzero. As a result, we can then define \(z(u, v) = \log w(u, v)\). It is easy to see that maximizing the reparametrized likelihood then maximizes the original likelihood.


\begin{remark}
\label{RemGeneralWeights}
The reasoning above allows us to apply our statistical theory to the case of general weights.
Specifically, letting \(\vect{e}_{i}\) denote the \(i^{\text{th}}\) standard basis vector in \(\reals^{m}\), we set \(\vect{x}_{u, v} = \vect{e}_{i}\) for \(i \cong (u, v)\).
We then have \(z(u, v) = \vect{e}_{i}^{T} \vect{\beta},\)
and substituting this into equation~\eqref{eqnGeneralOpt} yields a likelihood of the same form as equation~\eqref{eqnExpLikelihood}.
\end{remark}

Note that by Remark~\ref{RemGeneralWeights}, we have \(X = I_{m}\) in the case of general weights. Thus, Proposition~\ref{propCovRanks} implies that $\rank(C_t) = d - 1$, so the maximum likelihood estimator is not necessarily unique. This should not be surprising, however, since we already know the weights are only unique up to scaling. To amend this problem, suppose we fix the value of the last edge weight to remove one degree of freedom. Maximizing the likelihood then amounts to solving the maximum likelihood estimation problem with
\[
X
=
\left[\begin{array}{c}
I_{m - 1} \\ \hline 
\vect{0}^{T}
\end{array}
\right] \in \real^{m \times (m-1)}.
\]
By Proposition~\ref{propCovRanks}, we then have $\rank(C_{t}) = m - 1 = \rank(X)$, guaranteeing that the maximum likelihood estimator for the new likelihood function is unique.

We also remark that although the method for general weight estimation is mathematically rigorous and leads to accurate methods for simulating future spread of the disease, it may be more difficult to interpret the meaning of the estimated coefficients. This is unlike the case of the exponential weights parametrization, in which we can interpret the relative sizes of components in the parameter vector as providing the relative importance of various edge covariates.


Finally, we return to the setting of a $d$-dimensional exponential parametrization. Suppose we have obtained weight estimates $\tilde{\vect{w}}$ using Algorithm~\ref{algGeneralWeights}. To obtain an exponential parametrization, we simply perform the following projection, where the logarithm of $\tilde{\vect{w}}$ is taken componentwise:
\begin{equation}
\tilde{\vect{\beta}}
=
\arg\min_{\vect{\beta} \in \reals^{d}} \|X \vect{\beta} - \log \tilde{\vect{w}}\|_2^{2}
=
(X^{T} X)^{-1} X^{T} (\log \tilde{\vect{w}}) 
\defeq 
\proj_{X}(\tilde{\vect{w}}).
\label{eqnProjection}
\end{equation} 

Recall that \(\tilde{\vect{w}}\) is only unique up to scaling, so we may wonder about the consequences of this scaling on the projection. However, the arbitrary scale factor is accounted for by adding an intercept term to $\beta$ when taking the projection.
Indeed, note that scaling \(\tilde{\vect{w}}\) results in adding a constant to the coordinates of \(\log \tilde{\vect{w}}\), which is eliminated by the intercept term.

Naturally, accurate estimates of \(\tilde{\vect{w}}\) translate into accurate estimates of \(\tilde{\vect{\beta}}\). To rigorize this notion, we state the following result, which follows immediately from the continuity and measurability of \(\proj_{X}\):

\begin{proposition}
Let \(\vbeta_{0}\) be the true parameter of dimension \(d\),
and define \(\vect{w}_{0} = \exp(X \vbeta_{0})\).
If \(\tilde{\vect{w}}_{k}\) is consistent for \(\vect{w}_{0}\), 
then \(\tilde{\vbeta}_{k} = \proj_{X}(\tilde{\vect{w}}_{k})\) is consistent for \(\vbeta_{0}\).
\label{propProjection}
\end{proposition}
As a result, we have a suitable way of recovering the effects of covariates, provided we have a consistent estimator for general weights.


\subsection{Unknown sources}

We now consider an extension to the case where the set of infecting vertices is unknown. The probability of an infection set $\infected_k$ is then
\begin{align}
& \begin{aligned}
\prob\left(\infected_{k}\right) 
&=
\frac{1}{n} 
\cdot 
\frac{ \sum_{t = 1}^{1} b_2(v_t, v_2) w(v_{t}, v_{2})}{\sum_{e = 1}^{m} b_{2}(e) w(e)}
\cdots
\frac{\sum_{t = 1}^{k - 1} b_k(v_t, v_k) w(v_{t}, v_{k})}{\sum_{e = 1}^{m} b_{k}(e) w(e)}.
\label{eqnModel2}
   \end{aligned}
\end{align}
The process described in equation~\eqref{eqnModel2} also parallels the evolution of a generalized P\'{o}lya urn model, where the balls correspond to vertices rather than edges. Here, the replacement matrix \(W\) is \(n \times n\), and each entry \(W_{u, v}\) is simply equal to the weight function \(w(u, v)\).

However, we encounter one complication concerning the maximum likelihood estimator: the log-likelihood is no longer convex in general.
We mention a special case in which the theory described in the previous sections may be applied directly. Suppose the weight function \(w(u, v)\) is known to be constant in \(u\).
If we consider the exponential parametrization  \(w(u, v) = \vect{x}_{v}^{T} \vect{\beta}\) for some vectors \(\vect{x}_{v}\), we obtain the log-likelihood
\begin{align}
& \begin{aligned}
\ell(\vect{\beta}; \infected) 
&=
\sum_{t = 2}^{k} 
\left(\vect{x}_{v_{t}}^{T} \vect{\beta} 
+
\log \left(\sum_{i=1}^{t-1} b_{t}(v_i, v_t) \right)
\right)
- 
\sum_{t = 2}^{k} 
\log\left[
\sum_{e = 1}^{m} b_{t}(e) \exp(\vect{x}_{e}^{T} \vbeta)
\right]
- \log n.
\label{eqnExpLikelihoodVertices}
   \end{aligned}
\end{align}
The only difference between equations~\eqref{eqnExpLikelihood} and \eqref{eqnExpLikelihoodVertices} is the change of the constant term in the first summand, which does not alter the statistical analysis.


\section{Estimation without order information}
\label{secWithoutOrder}

\subsection{Problem setting}

Finally, we examine the substantially more difficult setting where edge transmissions are unordered: Instead of receiving information \(\ordered_{k}\), we observe  \(\unordered_{k} = \left\{(u_{2}, v_{2}), \ldots, (u_{k}, v_{k})\right\}\),
the unordered set of edges across which infection has spread.

We first note that the infection could have occurred over a variety of different paths, i.e.,
possible orderings of elements of \(\unordered_{k}\) that could produce an infection vector \(\ordered_{k}\). Since the computation of the likelihood involves summing over all possible infection paths, computing the likelihood will often be intractable for large graphs. For instance, for the complete graph, this leads to \((k - 1)!\) summands. In addition, the log-likelihood takes a more complicated form and may not be concave in general; more precisely, the log-likelihood results in the logarithm being applied to an additional sum, leading to a composition of a convex and a concave function. Since this is in general not concave, we do not automatically obtain an efficient method for computing the maximum likelihood estimator even if we ignore the intractability of the sums. Thus, we need an alternative solution. We consider the case of general weights, for which we offer two approaches.


\subsection{Using a limiting distribution}
\label{subsecEmpirical}

The first approach is to derive the unknown weight vector $\vect{w}$ as the limit of a quantity computed using finite samples. Since \(\vect{w}\) may be scaled arbitrarily, we set \(\|\vect{w}\|_{1} = 1\). For an edge \(i = (u, v)\), we write \(c_{t}(i)\) to denote the number of transmissions across edge \(i\) prior to time \(t\), leading to the vector $\vect{c}_t$. Denoting $f(\vect{v}, \vect{b}) \defeq \frac{\vect{v} \odiv \vect{b}}{\|\vect{v} \odiv \vect{b}\|_1}$, we have the following theorem:
\begin{theorem}
\label{thmEmpiricalWeights}
We have the almost-sure convergence
\[
\lim_{k \rightarrow \infty} f\left(\frac{\vect{c}_{k}}{\|\vect{c}_{k}\|_{1}}, \vect{b}_{k}\right) = \vect{w}.
\]
\end{theorem}
The proof of Theorem~\ref{thmEmpiricalWeights} is contained in Appendix~\ref{AppWithoutOrder}. A key fact is that 
$\vect{c}_k / \|\vect{c}_k\|_1 \asto \vect{\pi}_{\vbeta_0}$, almost surely, from P\'{o}lya urn theory.

This motivates the following algorithm. Note that the algorithm uses the vectors $\vect{b}_{k+1}$ and $\vect{c}_{k+1}$, which contain the cumulative infection information over all $k$ time steps.

\begin{algorithm}[!h]
\SetKwInOut{Input}{Input}
\Input{Vectors $\vect{b}_{k+1}$ and \(\vect{c}_{k + 1}\).}
Compute the renormalized vector $\tilde{\vect{c}} = \frac{\vect{c}_{k+1}}{\|\vect{c}_{k+1}\|_1}$.

Compute $\vect{w} = \frac{\tilde{\vect{c}} \odiv \vect{b}_{k+1}}{\|\tilde{\vect{c}} \odiv \vect{b}_{k+1}\|_1}$.

Return $\vect{w}$.
\caption{Weight estimate via limiting distribution}
\label{algLimitingEstimate}
\end{algorithm}

Unfortunately, a method for obtaining confidence intervals or performing other statistical inference procedures based on the estimates obtained by Algorithm~\ref{algLimitingEstimate} appears to be more complicated than in the case of the maximum likelihood estimator.


\subsection{Using a fixed-point equation}
\label{subsecFixedPointEquation}

The second method leverages the observation that the limiting distribution $\vect{\pi}_{\vbeta_0}$ is the leading left eigenvector of the weight matrix \(W(\vbeta_0)\), when normalized to be a probability distribution (cf.\ Lemma~\ref{lemmaUrnConvergence} in Appendix~\ref{appPolyaUrns}). Since the weight matrix \(W\) only depends on \(\vect{w}\) and the topology of \(\graphset\), which we assume is fixed,
we can denote the desired left eigenvector by \(\vect{\pi}_{\vect{w}}\). Furthermore, by Lemma~\ref{propPiWeights} in Appendix~\ref{appPolyaUrns}, the true weights must then satisfy
\begin{equation}
\vect{\pi}_{\vect{w}}
=
\frac{\vect{b}_{\infty} \circ \vect{w}}{\vect{b}_{\infty}^{T} \vect{w}},
\label{eqnPiFP}
\end{equation}
where $\vect{b}_\infty = \lim_{k \rightarrow \infty} \vect{b} / \|\vect{b}_k\|_1$.
One way to view equation~\eqref{eqnPiFP} is as a fixed-point equation, which we may attempt to solve in an iterative manner: Given some \(\vect{w}\), we can compute the weight matrix \(W\) and its eigenvector \(\vect{\pi}_{\vect{w}}\). By using \(\vect{b}_{\infty}\) (or a finite-sample estimate of this quantity), we can compute a weight vector \(\vect{w}'\) from equation~\eqref{eqnPiFP}. An algorithm based on this idea is provided in Algorithm~\ref{algIterativeEstimate}, where we use $\vect{v}_0$ to denote the initial guess for the distribution $\vect{\pi}$. 
Again, the algorithm uses the vector $\vect{b}_{k+1}$, which contains the cumulative infection information over all $k$ time steps.

\begin{algorithm}[!h]
\SetKwInOut{Input}{Input}
\Input{Vectors \(\vect{b}_{k+1}\) and \(\vect{v}_{0}\), number of iterations \(T\).}
\For{$i = 1, \ldots, T$}{
Compute \(\vect{w}_{i}\) from \(\vect{b}_{k+1}\) and \(\vect{v}_{i - 1}\) using Theorem~\ref{thmEmpiricalWeights}: $\vect{w}_i = \frac{\vect{v}_{i-1} \odiv \vect{b}_{k+1}}{\|\vect{v}_{i-1} \odiv \vect{b}_{k+1}\|_1}$.

Compute the matrix $W_i = W(\vect{w}_{i})$.

Compute the normalized leading left eigenvector \(\vect{\pi}_{\vect{w}_{i}}\) of \(W_i\) and define this to be \(\vect{v}_{i}\).}

Return $\frac{\vect{v}_T \odiv \vect{b}_{k+1}}{\|\vect{v}_T \odiv \vect{b}_{k+1}\|_1}$.
\caption{Iterative solution to equation~\eqref{eqnPiFP}}
\label{algIterativeEstimate}
\end{algorithm}

Some natural questions are whether equation~\eqref{eqnPiFP} always has a fixed point with good statistical properties, and whether Algorithm~\ref{algIterativeEstimate} converges to such a fixed point. Unfortunately, such an analysis appears to be beyond the scope of this paper. A simulation study is provided in the experiments of Section~\ref{secSims}.

A natural initialization for Algorithm~\ref{algIterativeEstimate} is the empirical proportion of edge transmission counts \(\vect{v}_0 = \vect{c}_{k+1} / \|\vect{c}_{k+1}\|_{1}\). In fact, we can show that this initialization gives rise to a fixed point in a specific case:


\begin{theorem}
Let \(\graphset\) be the directed cyclic graph on \(n\) vertices,
and consider the associated urn scheme.
The empirical distribution
\[
\vect{\pi}_k
=
\frac{\vect{c}_{k + 1}}{\|\vect{c}_{k + 1}\|_{1}}
\]
leads to a fixed point of the equation
\[
\vect{\pi}_{\vect{w}}
=
\frac{(\vect{b}_{k + 1} - \vect{b}_{2}) \circ \vect{w}}{(\vect{b}_{k + 1} - \vect{b}_{2})^{T} \vect{w}},
\]
where the weights are given by
\begin{align*}
\vect{w} & = f(\vect{\pi}_k, \vect{b}_{k+1} - \vect{b}_2) \\
& =
\left(
\frac{b_{k + 1}(2) - b_{2}(2)}{b_{k + 1}(1) - b_{2}(1)}, \ldots, \frac{b_{k + 1}(n) - b_{2}(n)}{b_{k + 1}(n - 1) - b_{2}(n - 1)}, \, \frac{b_{k + 1}(1) - b_{2}(1)}{b_{k + 1}(n) - b_{2}(n)}\right)^T.
\end{align*}
\label{theoremCyclic}
\end{theorem}

The proof of Theorem~\ref{theoremCyclic} is contained in Appendix~\ref{subsecWOOProofs}. Note that as $k \rightarrow \infty$, this means $\vect{\pi}_k$ is an approximate fixed point of equation~\eqref{eqnPiFP}.

\begin{remark}
The key fact underlying the analysis of Theorem~\ref{theoremCyclic} is that \(b_{k + 1}(i + 1) - b_{2}(i + 1)\) counts the number of times that infection has spread across edge \((i, i + 1)\), since subtracting \(b_{2}(i + 1)\) merely corresponds to ignoring the ball (edge) used to initialize the process. Thus, we have \(c_{k + 1}(i) = b_{k + 1}(i + 1) - b_{2}(i + 1)\). In other graphs besides the cyclic graph, such a simple relation may not hold.
\end{remark}

Theorem~\ref{theoremCyclic} provides one example where our separate intuitions for using the empirical distribution of infection spreads and for finding a solution to our fixed point equation nearly coincide.


\section{Simulations}
\label{secSims}

To test our methods, we conducted simulations using a variety of parameter settings.
We considered network topologies corresponding to a directed cycle graph of various sizes, both with and without self-loops on the vertices. We took \(\vbeta_0\) to be equal to the \(d\)-dimensional vector of ones in the case of no self-loops, and the \((d + 1)\)-dimensional vector of ones in the case of self-loops, for different values of $d$. We also varied the infection count $k$. For each set of parameters, we simulated several infections using edge weights drawn independently and identically from a normal distribution, with the exception of self-loop parameters. We then calculated the empirical distribution estimator, fixed point estimator, maximum likelihood estimator, and general weights estimator for each simulated infection.
In many cases, the fixed point estimator could not be calculated because of the nonexistence of the required Perron eigenvector, which is only guaranteed to emerge for sufficiently large $k$; thus, most of the comparisons we report are between the other three estimators.

The results of the simulations, together with full details of the implementations, are reported in Appendix~\ref{appSimulationResults}. Tables~\ref{tableCycle5075RMSE}-\ref{tableCycle5075time} provide results for the cycle graph without self-loops; results for the cycle graph with self-loops are provided in Tables~\ref{tableCycleLoops5075RMSE}-\ref{tableCycleLoops5075time}. Table~\ref{tableCycle5075RMSE} reports the accuracy of our estimates, measured according to the root mean squared error between the estimate \(\hat{\vbeta}\) and the true \(\vbeta_0\). As discussed in Section~\ref{SecGeneral}, for the methods other than maximum likelihood, we compute the error after projecting our estimates of the edge weights back into \(\reals^{d}\). In general, the maximum likelihood estimator performs the best, which is unsurprising.

Next, we compare the performance of the asymptotic confidence intervals given by the maximum likelihood estimator. We approximate the asymptotic information \(I_{\infty}(\vbeta_0)\) by the empirical average
\begin{equation}
\hat{I}_{\infty}(\hat{\vbeta}_{k})
=
\frac{1}{k} \sum_{t = 2}^{k} I_{t}(\hat{\vbeta}_{k}),
\label{eqnInformationEstimate}
\end{equation}
which is known to converge almost surely to \(I_{\infty}(\vbeta_0)\). However, in some cases---particularly those of high \(n\) and \(d\) or low \(k\)---the matrix $\hat{I}_{\infty}(\hat{\vbeta}_k)$ is nearly singular, so inverting the matrix to compute confidence intervals results in numerical errors.
In order to gauge the presence of numerical errors, we compute the proportion of runs in which negative diagonal entries are present in the calculated inverse of $\hat{I}_{\infty}(\hat{\vbeta}_k)$, which is supposed to be positive semidefinite.
We also provide the proportion of confidence intervals that cover the true \(\vbeta_0\). Finally, we compute the average length of confidence intervals and (for comparison) the average length of a centered interval that would be necessary to cover \(\vbeta_0\) in \(95\%\) of the simulations.

As seen in Tables~\ref{tableCycle50CI} and~\ref{tableCycle75CI}, the empirical coverage of our computed confidence intervals is noticeably smaller than the target level. This is likely due to the instability of ball count proportions in a P\'{o}lya urn after relatively few draws. Indeed, the theory for our asymptotic intervals relies on almost sure convergence of the P\'{o}lya urn process to a particular limiting distribution, but for moderate values of $k$, the contents of the urn may be quite far from the asymptotic limit.
More concretely, suppose the limit were only approximately realized for values \(k > K\).
Then the estimator~\eqref{eqnInformationEstimate} might be more accurately computed by taking an average of the latter $k-K$ terms, and the scale factor of  \(k^{1/2}\) in the formula for asymptotic normality in Theorem~\ref{theoremAsNormal} should analogously be replaced by \((k - K)^{1/2}\). Of course, this observation does not contradict our theory, since \(k^{-1/2}(k - K)^{1/2} \) converges to \(1\) as $k$ tends to infinity. In practice, one might want to consider some ``burn-in" \(K\) when computing the confidence intervals. Alternatively, other estimates of the inverse information matrix that avoid the numerical accuracies issues arising from the expression in equation~\eqref{eqnInformationEstimate} might provide more accurate confidence intervals.

The final item of interest is computation time, provided in Table~\ref{tableCycle5075time}. In general, the maximum likelihood estimator takes the longest time to compute; this is particularly noticeable for large values of \(n\), \(d\), and \(k\). In contrast to other methods, the computation time for the empirical distribution estimator and the general weights estimator does not appear to depend too much on the dimension \(d\). All simulations and calculations were conducted in Python; the maximum likelihood estimator and general weights computation used the SCS solver of CVXPY.

Finally, we ran experiments to compare the performance of Algorithms~\ref{algLimitingEstimate} and~\ref{algIterativeEstimate} in the absence of infection order information. The results are provided in Table~\ref{tableCycle2RMSE}. The simulations were conducted on a cycle graph of size \(n = 2\) with \(d = 1\), with different values of $k$. As noted in the simulation results, the fixed point estimator performs similarly to the empirical distribution estimator. The fixed point estimator was calculated with \(5\) iterations, which is likely sufficient for this simple case, since one can easily show that a unique fixed point exists in this case and the iterative algorithm exhibits linear convergence to this fixed point.


\section{Application to Ebola Spread}
\label{secApplications}

In this section, we describe an application of our methods to data collected from the spread of Ebola in Guinea, Sierra Leone, and Liberia from 2013-2015. The official count for the epidemic included 28,616 cases suspected or confirmed, 15,227 cases confirmed, and 11,310 total deaths \citep{CDC2016}. The dataset we analyzed contains genome sequences from 3,219 infected individuals, which allowed researchers to infer transmissions between cities and states in the region. Specifically, epidemiologists applied a Bayesian phylogeographic model to infer relationships between sequences determine parameters related to the spread of infection, concluding that there was strong evidence for the importance of five variables \citep{dudas2017}.

Due to the uncertainty in inferring phylogenetic trees, a number of possible transmission patterns were produced by the model. In our work, we considered one particular realization, with the goal of checking whether the transmission structure provides evidence for the importance of the same predictors according to the model. The transmission pattern is shown in Figure~\ref{figEbolaBW}. It consists of 319 inter-regional transmissions between 46 of the 63 regions. Of the 319 transmissions, 166 occurred between unique source-destination pairs.

\begin{figure}[!h]
\centering 
\captionsetup{width=0.8\linewidth}
\includegraphics[width=4.5in]{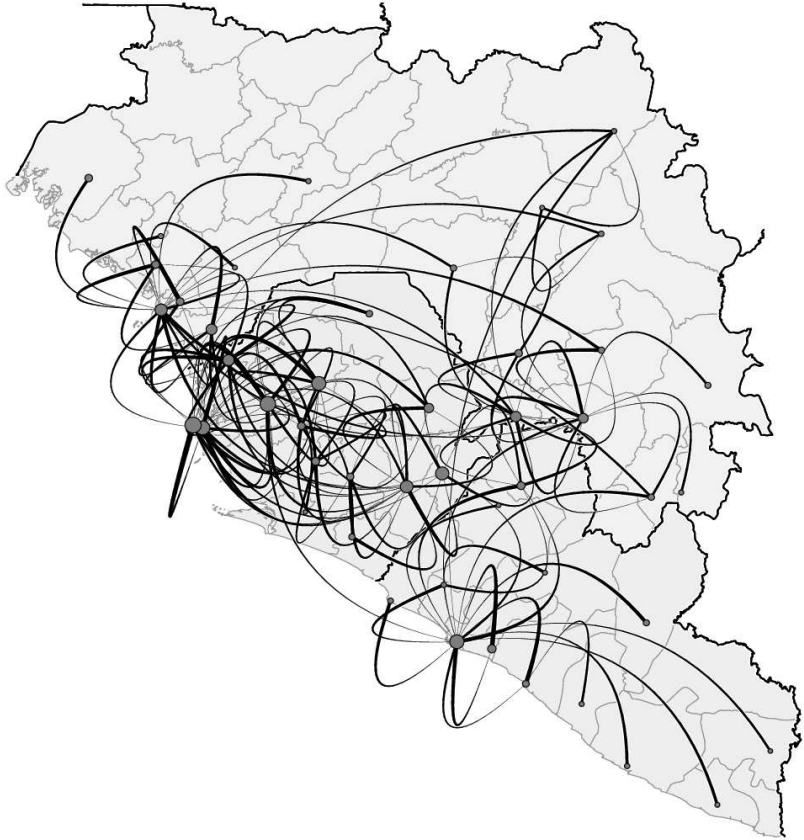}
\caption{\citep{dudas2017} The spread of Ebola in West Africa from December 2013 to November 2015. Bold lines are international borders; gray lines are regional borders. Circles are positioned at population centroids of regions affected with Ebola, and larger circles indicate more Ebola cases in the region. Arcs between circles represent inter-regional transmissions. Larger arcs represent more transmissions; the source region is the narrower end of the arc.}
\label{figEbolaBW}
\end{figure}

Following \cite{dudas2017}, we used 27 explanatory variables for the edges.
A description of the covariates is supplied in Tables~\ref{tableEbolaCovariates1} and~\ref{tableEbolaCovariates2}. We applied our model to the inter-regional transmission data, using a complete graph topology for connections between regions. Computing the maximum likelihood estimator took 5,715 minutes. The estimator, standard errors, and \(t\)-statistics are provided in Table~\ref{tableEbolaResults}. One difficulty in interpreting the results is due to the small estimated standard error. This is consistent with our simulations in Section~\ref{secSims}, where often the proper coverage is not attained. As a result, we speculate that the standard errors should all be increased by a large multiplicative factor to improve accuracy of inference.

The largest absolute \(t\)-statistics correspond to the great circle distance between two regions; the destination population; and the source population. In addition, the presence of an international border and the various international boundaries had large absolute \(t\)-statistics. This is consistent with the analysis of \cite{dudas2017}. However, there are a few notable differences: First, our model found the ``source t.t.\ 100k'' variable and ``source prec.'' variables to be relatively important. Second, one of the least significant variables for our model is the source temperature seasonality, which was regarded as relatively important for \cite{dudas2017}. Our hypothesis is that none of these three variables are particularly significant.


\section{Discussion}
\label{secDiscuss}

This work leaves a number of directions for future inquiry. The most important practical question to address is to improve the finite-sample accuracy of our confidence intervals. As demonstrated in experiments, the standard errors should be larger to ensure better coverage. As mentioned earlier, one could consider some burn-in \(K\) when computing confidence intervals;
however, this may only be feasible if the number of infection events $k$ is reasonably large. A better solution would be to bound the nonasymptotic bias in the urn process more explicitly.

In the case when the identities of infecting vertices are unknown, we have only studied the situation where the weight functions \(w(u, v)\) are constant in \(u\), which corresponds to the scenario where all neighbors similarly affect the odds of subsequent infections at \(u\).
This is a restrictive assumption, since it precludes the estimation of parameters based on observable edge covariates. Another setting of interest involves incomplete source information: In online social networks, it is sometimes the case that certain transmission edges are known explicitly through direct interactions, while others are not. Hence, a model and analysis that can accommodate a mix of infecting-infected pairs \((u_{t}, v_{t})\) pairs and standalone \(v_{t}\) could be useful.

Finally, linking to preexisting literature, introducing the P\'{o}lya urn model or a continuous-time variant in a Bayesian setting might be of interest. The goal would be to integrate the model into phylogenetic models as in \cite{dudas2017} or \cite{lemey2014}, rather than performing a two-step process as employed in our experiments, where the first step consists of segmenting the infection process. This would allow us to conduct valid statistical inference without first conditioning on the correctness of the pre-estimated network.



\section*{Acknowledgement}
The authors thank Varun Jog for helpful discussions while preparing this paper.

\appendix

\section{Proof sketches}
\label{secProofs}

\subsection{Lemmas on P\'{o}lya urns}
\label{appPolyaUrns}

In Section~\ref{subsecPolya}, we discussed the relationship between our contagion model and P\'{o}lya urns. Here, we state various key lemmas which are used throughout our analysis. Proofs are contained in Appendix~\ref{AppPolyaProofs}.

The analysis of this section applies to {\em irreducible} P\'{o}lya urns, meaning that the matrix \(W\) is irreducible. Equivalently, the matrix \(\exp(t W)\) has strictly positive entries for $t > 0$.
Our first task is to prove that the urns arising in our problem setting are indeed irreducible.

\begin{lemma}
If \(\graphset\) is strongly connected, the generalized P\'{o}lya urn defined by the replacement matrix \(W\) is irreducible.
\label{propNonTriangular}
\end{lemma}

Next, we have the following lemma:

\begin{lemma}
Suppose the urn \(W(\vbeta_0)\) is irreducible.
The distribution \(\prob_{\vbeta_{0}, t}\) converges almost surely to the unique strictly positive distribution \(\vect{\pi}_{\vbeta_0}\) given by the leading left eigenvector of \(W(\vbeta_{0})\).
\label{lemmaUrnConvergence}
\end{lemma}

Note that when we say that the distribution converges to a deterministic vector, we identify the vector as the non-random measure over the $m$ colors, where the \(i^\text{th}\) index of the vector is the probability that the next ball drawn is of color \(i\).

We also have the following result:
\begin{lemma}
There exists a deterministic vector $\vect{b}_\infty$ with components  
\(\vect{b}_{\infty}(i) 
\defeq 
\lim_{t \rightarrow \infty} \vect{b}_{t}(i) / \|\vect{b}_{t}\|_{1}$ 
such that the distribution \(\vect{\pi}_{\vbeta_0}\) satisfies the equation
\begin{equation}
\vect{\pi}_{\vbeta_0}
=
\frac{\vect{b}_{\infty} \circ \vect{w}}{\vect{b}_{\infty}^{T} \vect{w}}.
\label{eqnPiWeights}
\end{equation}
\label{propPiWeights}
\end{lemma}


\subsection{Asymptotic theory}
\label{SecAsymp}

We borrow the general proof strategy from \cite[Chapter~6, Theorem~5.1]{lehmann2006} and adapt it to our setting.

To prove consistency, we wish to show that there is a sequence of maximizers of the log-likelihood that converge to the true \(\vbeta\).
We want to show that \(\ell(\vbeta) < \ell(\vbeta_{0})\) for all \(\vbeta\) in the spherical shell of radius \(r\) around \(\vbeta_{0}\).
Then, by the continuity of the log-likelihood,
the maximizer \(\hat{\vbeta}_{k}\) must be within the interior of the ball of radius \(r\) centered at \(\vbeta_{0}\). This is summarized in the following lemma, proved in Appendix~\ref{AppLemBall}:

\begin{lemma}
\label{LemBall}
Let
\begin{equation*}
S_{k}(r) 
\defeq
\Big\{\ordered_{k}: \ell(\vbeta_{0} + \vect{y}; \ordered_{k}) < \ell(\vbeta_{0}; \ordered_{k}) \text{ where }
\|\vect{y}\|_2 = r \Big\}.
\end{equation*}
There exists $r_0 > 0$ such that for all $r \le r_0$, we have
\begin{equation*}
\lim_{k \rightarrow \infty} \prob_{\vbeta_0}\big(S_k(r)\big) = 1.
\end{equation*}
\end{lemma}

If \(\ordered_{k}\) is in \(S_{k}(r)\),
then by the continuity of the log-likelihood, there is a \(\hat{\vbeta}_{k}(r)\) in the ball of radius \(r\) centered at \(\vbeta_{0}\) that is a local maximum.
By the concavity of the log-likelihood, this is also the unique global maximum $\hat{\vbeta}_k^*$. Hence, 
\begin{equation*}
\prob_{\vbeta_0} \left(\|\hat{\vbeta}_k^* - \vbeta_0\|_2 \ge r\right) \le 1 - \prob_{\vbeta_0} \big(S_k(r)\big).
\end{equation*}
By Lemma~\ref{LemBall}, the limit of the right-hand side as $k \rightarrow \infty$ is equal to 0. Thus, we clearly have
\begin{equation*}
\lim_{k \rightarrow \infty} \prob_{\vbeta_0} \left(\|\hat{\vbeta}_k^* - \vbeta_0\|_2 \ge r\right) = 0,
\end{equation*}
as well. Hence, we have the convergence $\hat{\vbeta}_k^* \probto \vbeta_0$ in probability, implying consistency.


We now turn to the proof of asymptotic normality. We first expand the coordinatewise derivatives of the log-likelihood about \(\vbeta_{0}\) to obtain
\begin{align*}
\ell'_{a}(\vbeta) & = \ell'_{a}(\vbeta_{0}) + \sum_{b = 1}^{d} (\sbeta(b) - \sbeta_{0}(b)) \ell''_{ab}(\vbeta_{0}) \\
& \qquad \qquad + \frac{1}{2} \sum_{b = 1}^{d} \sum_{c = 1}^{d} (\sbeta(b) - \sbeta_{0}(b))(\sbeta(c) - \sbeta_{0}(c)) \ell'''_{abc}(\vbeta^{*}),
\end{align*}
for some \(\vbeta^{*}\) on the line segment between \(\vbeta\) and \(\vbeta_{0}\), where we have used the shorthand notation $\ell'_a(\vbeta)$ to denote the $a^{\text{th}}$ coordinate of the gradient of $\ell$, and similarly for the higher-order derivatives.

Substituting \(\vbeta = \hat{\vbeta}_{k}\) sets the left hand side to \(0\).
Rearranging, we then obtain
\begin{align*}
& \begin{aligned}
\sqrt{k} \sum_{b = 1}^{d} (\hat{\sbeta}_{k}(b) - \sbeta_{0}(b)) 
\left(\frac{1}{k} \ell''_{ab} (\vbeta_{0}) + \frac{1}{2k} \sum_{c = 1}^{d} (\hat{\sbeta}_{k}(c) - \sbeta_{0}(c)) \ell'''_{abc}(\vbeta^{*})\right)
&= 
-\frac{1}{\sqrt{k}} \ell'_{a}(\vbeta_{0}).
\end{aligned}
\end{align*}
We will apply Lemma~\ref{lemmaSolveLinear} with
\begin{align*}
& \begin{aligned}
Y_{b, k} 
&=
\sqrt{k} \left(\hat{\sbeta}_{k}(b) - \sbeta_{0}(b) \right), \\ 
A_{ab, k}
&=
\frac{1}{k} \ell''_{ab}(\vbeta_{0}) 
+ \frac{1}{2k} \sum_{c = 1}^{d} \left(\hat{\sbeta}_{k}(c) - \sbeta_{0}(c)\right) \ell'''_{abc}(\vbeta^{*}), \\ 
T_{a, k} 
&= 
- \frac{1}{\sqrt{k}} \ell'_{a}(\vbeta_{0}).
\end{aligned}
\end{align*}

We have the following lemma, proved in Appendix~\ref{AppLemConverge}:
\begin{lemma}
\label{LemConverge}
We have the convergences
\begin{align*}
T_k & \lawto \normal(\vect{0}, I_{\infty}(\vbeta_{0})), \\
A_{k} & \probto -I_{\infty}(\vbeta_{0})
\end{align*}
in distribution and in probability, respectively.
\end{lemma}

Putting this all together means that 
\[
Y
\sim 
\normal_{d}\left(\vect{0}, I_{\infty}^{-1}(\vbeta_{0})\right).
\]
This completes the proof of asymptotic normality.


\section{Proofs for Section~\ref{appPolyaUrns}}
\label{AppPolyaProofs}

\subsection{Proof of Lemma~\ref{propNonTriangular}}

Consider two edges $e = (u,v)$ and $f = (x,y)$. Since $\graphset$ is strongly connected, there exists a path of some length $r$ from $e$ to $f$. Hence, the term $t^r W^r / (r!)$ in the expansion of $\exp(tW)$ has entry $(e,f)$ strictly positive. Since all entries of $W$ are nonnegative, it is clear that the $(e,f)$ entry of $\exp(tW)$ must be strictly positive, as well.

\subsection{Proof of Lemma~\ref{lemmaUrnConvergence}}

Let \(X_{t}\) be a continuous-time branching process, where in the time interval \((t, t + dt]\), each particle of type \(i\) branches independently with probability \(w_{\vbeta_{0}}(i) dt\).
When a particle of type \(i\) branches, it is replaced by one particle of type \(i\) and one particle of each type \(j\) for which \(W(\vbeta_{0})_{ij} > 0\).
Let \(\vect{\pi}_{\vbeta_0}\) be the leading left eigenvector of \(W(\vbeta_{0})\), normalized to be a probability distribution, 
and let \(\lambda_{1}\) be the associated eigenvalue. 
Note that $\vect{\pi}_{\vbeta_0}$ is unique by the Perron-Frobenius theorem.
Define \(X'_{t} = w_{\vbeta_{0}} \circ X_{t}\).
By Theorem 1 of \cite{athreya1968}, we have the almost sure convergence of 
\(X'_{t}\exp(-t \lambda_{1})\) to  
\(X_{\infty} \vect{\pi}_{\vbeta_0}\), where \(X_{\infty}\) is a nonnegative random variable.
For completeness, we provide a statement of this theorem in Appendix~\ref{appOtherTheorems}.

By normalizing, we obtain the limiting distribution 
\[
\lim_{t \to \infty}
\frac{X'_{t} \exp(-t \lambda_{1})}{\|X'_{t} \exp(-t \lambda_{1})\|_{1}}
=
\vect{\pi}_{\vbeta_0},
\]
so it suffices to show that we can recover our urn process from the branching process via stopping times.
For this, we follow Theorem 1 of \cite{athreya1968embedding} and the ensuing discussion.

Let \(\tau_{K}\) be the \(K^\text{th}\) split time of the process \(X_{t}\).
Then,  \(\tau_{1}, \ldots, \tau_{K}\) forms an increasing sequence of stopping times.
At \(\tau_{K}\), there are \(X_{\tau_{K}}(i)\) particles of type \(i\).
Thus, the probability that the next split is of a particle of type \(i\) is
\[
\frac{w_{\vbeta_{0}}(i) X_{\tau_{K}}(i)}{\sum_{j = 1}^{m} w_{\vbeta_{0}}(j) X_{\tau_{K}}(j)}
=
\frac{X'_{\tau_{K}}(i)}{\sum_{j = 1}^{m} X'_{\tau_{K}}(j)},
\]
since all particles of a type \(j\) have independently-distributed exponential lifetimes of mean \(1 / w_{\vbeta_{0}}(j)\).
The resulting change in \(X_{t}\) is that one particle of type \(j\) is added for each \(j\) such that \(W(\vbeta_{0})_{ij} > 0\).

We now consider the analogous urn process \(Y_{K}\), which has replacement weight matrix \(W(\vbeta_{0})\).
Unlike in standard urn theory, we distinguish between the number of balls and their weights;
when we draw a ball of color \(i\), we replace it with a ball of weight \(W(\vbeta_{0})_{ij} = w_{\vbeta_{0}}(j)\), for each \(j\) for which \(w_{\vbeta_{0}} > 0\).
Again, at the \(K^\text{th}\) time, we have \(Y_{K}(i)\) balls of type \(i\) in the urn, 
and on the next draw, the probability we draw a ball of type \(i\) is 
\[
\frac{w_{\vbeta_{0}}(i) Y_{K}(i)}{\sum_{j = 1}^{m} w_{\vbeta_{0}}(j) Y_{K}(j)}.
\]
Note that the ensuing replacement is then the same as for \(X_{t}\).

Thus, the transition probabilities of \(X_{t}\) and \(Y_{t}\) are the same.
Additionally, since both processes are Markov, they are equivalent.
Since \(\tau_{K} \to \infty\) as \(K \to \infty\),
we have
\[
\prob _{\vbeta_{0}, K}(i)
=
\frac{w_{\vbeta_{0}}(i) Y_{K}(i)}{\sum_{j = 1}^{m} w_{\vbeta_{0}}(j) Y_{K}(j)}
=
\frac{w_{\vbeta_{0}}(i) X_{\tau_{K}}(i)}{\sum_{j = 1}^{m} w_{\vbeta_{0}}(j) X_{\tau_{K}}(j)}
=
\frac{X'_{\tau_{K}}(i)}{\sum_{j = 1}^{m} X'_{\tau_{K}}(j)}.
\] 
By taking limits, we see that \(\prob_{\vbeta_{0}, \infty} = \vect{\pi}_{\vbeta_0}\), as required.


\subsection{Proof of Lemma~\ref{propPiWeights}}

We first show that the limiting vector $\vect{b}_\infty$ exists. We define
\[
b'_{t}(i)
\defeq
\frac{b_{t}(i)}{b_{t}(m)} \\ 
=
\frac{w_{\vbeta_{0}}(m)}{w_{\vbeta_{0}}(i)} \cdot \frac{\prob_{\vbeta_{0}, t}(i)}{\prob_{\vbeta_{0}, t}(m)}.
\]
By Lemma~\ref{lemmaUrnConvergence}, the distribution \(\prob_{\vbeta_{0}, \infty}\) is strictly positive, almost surely. Hence, for \(t\) sufficiently large, \(\prob_{\vbeta_{0}, t}\) is also strictly positive. Taking limits, we therefore obtain
\begin{align}
\label{EqnChoco}
\lim_{t \to \infty} b'_{t}(i)
&=
\frac{w_{\vbeta_{0}}(m)}{w_{\vbeta_{0}}(i)} \cdot \frac{\prob_{\vbeta_{0}, \infty}(i)}{\prob_{\vbeta_{0}, \infty}(m)}  
\defeq b_\infty'(i),
\end{align}
where $\vect{b}_\infty'$ is deterministic. In particular, we have
\begin{equation*}
\lim_{t \rightarrow \infty} \frac{b_t(i)}{\|\vect{b}_t\|_1} 
= \lim_{t \rightarrow \infty} \frac{b'_t(i)}{\|\vect{b}'_t\|_1} 
= \frac{\lim_{t \rightarrow \infty} b'_t(i)}{\lim_{t \rightarrow \infty} \|\vect{b}'_t\|_1} 
= \frac{b'_\infty(i)}{\|\vect{b}'_\infty\|_1}.
\end{equation*}
Accordingly, we define the vector $\vect{b}_\infty = \vect{b}'_\infty / \|\vect{b}'_\infty\|_1$ to obtain the desired result.

Now note that by equation~\eqref{EqnChoco}, we have
\[
\vect{b}'_{\infty} \; \propto \; \left(\prob_{\vbeta_{0}, \infty} \odiv \, \vect{w}\right).
\]
As a result, we have
\[
\frac{\vect{b}_{\infty} \circ \vect{w}}{\vect{b}_{\infty}^{T} \vect{w}} = \prob_{\vbeta_{0}, \infty},
\]
and by Lemma~\ref{lemmaUrnConvergence}, this is equivalent to \(\vect{\pi}_{\vbeta_0}\), almost surely.


\section{Proofs of existence and uniqueness lemmas}
\label{appExistUnique}

\subsection{Proof of Theorem~\ref{propExistenceConditions}}
\label{subsecExistenceComputation}

In this section, we consider basic proofs of existence and computation for the maximum likelihood estimator.

We will use both the likelihood \(L\) and the log-likelihood \(\ell\) as is convenient.
We begin with a useful lemma:
\begin{lemma}
Given \(X\), the log-likelihood $\ell(\vect{\beta}; \ordered_k)$ has a maximum
if and only if
for all \(\vbeta\),
there is no \(\vect{v}\) such that 
\(f_{\vbeta, \vect{v}}(t) \defeq \ell(\vbeta + t \vect{v}; \ordered_k)\)
is strictly increasing.
\label{lemmaExistenceRays}
\end{lemma}
\begin{proof}
The forward direction is simple, so we focus on the reverse direction.
First note that since the log-likelihood is concave and continuously differentiable, 
every \(f_{\vbeta, \vect{v}}\) is also concave and continuously differentiable.
As a result, the function \(f_{\vbeta, \vect{v}}\) is either constant, strictly decreasing, 
increasing to a maximum and then decreasing, or strictly increasing.
Let \(M_{\vbeta}(\vect{v}) = \max_{t \geq 0} f_{\vbeta, \vect{v}}(t)\).
Since \(f_{\vbeta, \vect{v}}\) is not strictly increasing, the function
\(M_{\vbeta}(\vect{v})\) is defined for all \(\vect{v}\). Furthermore, the vector $\vect{v}^*$ defined as $\arg\max_{\|\vect{v}\|_2 \le 1} M_{\vect{\beta}}(\vect{v})$ must exist, and that the maximum of $\ell$ must then agree with the maximum of $f_{\vbeta, \vect{v}^*}$.
\end{proof}


Returning to the proof of the theorem, suppose \(t\) and \(f\) exist such that \(b_{t}(f) > 0\) and \(\vect{x}_{e_{t}}^{\trans} \vect{v} < \vect{x}_{f}^{\trans} \vect{v}\).
The \(t^{\text{th}}\) term in the expansion of \(L(\vbeta + s \vect{v}; \ordered_k)\) is bounded above by
\[
p_{t}
=
\frac{b_{t}(e_{t}) \exp\left(\vect{x}_{e_{t}}^{T} (\vbeta + s\vect{v})\right)}{b_{t}(e_{t}) \exp\left(\vect{x}_{e_{t}}^{T} (\vbeta + s\vect{v})\right)
+
b_{t}(f) \exp\left(\vect{x}_{f}^{T} (\vbeta + s\vect{v})\right)}.
\]
Furthermore, since $\vect{x}_{e_t}^T \vect{v} < \vect{x}_f^T \vect{v}$, it is easy to see that \(p_{t} \rightarrow 0\) as $s \rightarrow \infty$. Since \(p_{t}\) bounds the likelihood,
we conclude that \(f_{\vbeta, \vect{v}}\) cannot be strictly increasing.
By Lemma~\ref{lemmaExistenceRays}, the log-likelihood cannot attain a maximum. Now suppose that for all $t$ and $f$ such that $b_t(f) > 0$, we have $\vect{x}_f^T \vect{v} = c$. Then for any $\vbeta$ and $\vect{v}$, the function $f_{\vbeta, \vect{v}}$ is constant, since each term is multiplied by $\exp(t\vect{x}_f^T \vect{v}) = \exp(tc)$ in the likelihood expression. In particular, this holds for any potential maximum $\hat{\vbeta}$, which is a contradiction. This proves the forward direction.

For the reverse direction, suppose that for some \(\vect{v}\), we have 
\(\vect{x}_{f}^{\trans} \vect{v} \leq \vect{x}_{e_{t}}^{\trans} \vect{v}\) for all \(f\) and all $t$ such that \(b_{t}(f) > 0\), with strict inequality for some \(f\). Recalling that
\begin{align*}
f_{\vbeta, \vect{v}}(s) & = \frac{b_t(e_t) \exp\left(\vect{x}_{e_t}^T (\vbeta + s\vect{v})\right)}{\sum_{e \in \edgeset} b_t(e) \exp\left(\vect{x}_e^T(\vbeta + s\vect{v})\right)},
\end{align*}
we may compute
\begin{align*}
f'_{\vbeta, \vect{v}}(s) 
& = \frac{\sum_{e \in \edgeset} b_t(e) \exp\left(\vect{x}_e^T(\vbeta + s\vect{v})\right) b_t(e_t) \exp\left(\vect{x}_{e_t}^T (\vbeta + s\vect{v})\right) \vect{x}_{e_t}^T \vect{v}}{{\left(\sum_{e \in \edgeset} b_t(e) \exp\left(\vect{x}_e^T(\vbeta + s\vect{v})\right)\right)^2}} \\
& \qquad - \frac{b_t(e_t)\exp\left(\vect{x}_{e_t}^T (\vbeta + s\vect{v})\right)\left(\sum_{e \in \edgeset} b_t(e) \exp\left(\vect{x}_e^T(\vbeta + s\vect{v})\right) \vect{x}_e^T \vect{v} \right)}{\left(\sum_{e \in \edgeset} b_t(e) \exp\left(\vect{x}_e^T(\vbeta + s\vect{v})\right)\right)^2} \\
& > \frac{\sum_{e \in \edgeset} b_t(e) \exp\left(\vect{x}_e^T(\vbeta + s\vect{v})\right) b_t(e_t) \exp\left(\vect{x}_{e_t}^T (\vbeta + s\vect{v})\right) \vect{x}_{e_t}^T \vect{v}}{{\left(\sum_{e \in \edgeset} b_t(e) \exp\left(\vect{x}_e^T(\vbeta + s\vect{v})\right)\right)^2}} \\
& \qquad - \frac{b_t(e_t)\exp\left(\vect{x}_{e_t}^T (\vbeta + s\vect{v})\right)\left(\sum_{e \in \edgeset} b_t(e) \exp\left(\vect{x}_e^T(\vbeta + s\vect{v})\right) \vect{x}_{e_t}^T \vect{v}\right)}{\left(\sum_{e \in \edgeset} b_t(e) \exp\left(\vect{x}_e^T(\vbeta + s\vect{v})\right)\right)^2} \\
& = 0,
\end{align*}
so \(f_{\vbeta, \vect{v}}\) is strictly increasing for each $\vect{\beta}$. Hence, by Lemma~\ref{lemmaExistenceRays}, the log-likelihood does not achieve a maximum. This proves the theorem.


\subsection{Proof of Proposition~\ref{propLogLikelihoodHessian}}
\label{appCovRanks}

Using Lemma~\ref{lemmaDiffMoment} and Remark~\ref{remarkDerivative} in Appendix~\ref{AppComp},
we have
\begin{align*}
\frac{\partial^{2}}{\partial \sbeta(a) \sbeta(b)} \ell(\vbeta; \ordered_{k}) & =
-\sum_{t = 2}^{k} \left(\expect_{t}[Z_{t}(a) Z_{t}(b)] - \expect_{t}[Z_{t}(a)] \expect_{t}[Z_{t}(b)]\right) \\
& = - \sum_{t = 2}^{k} \cov_{t}(Z_{t}(a), Z_{t}(b)).
\end{align*}
This completes the proof.

\subsection{Proof of Proposition~\ref{propCovRanks}}

We begin with a useful lemma:

\begin{lemma}
Let \(M\) be an invertible \(d \times d\) matrix, and let \(X' = XM\). Let $Z_t$ denote a random vector obtained by sampling the rows of $X$ according to a fixed probability distribution where all rows are sampled with strictly positive probability, and let $Z'_t$ denote a random vector obtained by sampling the rows of $X'$ according to the same distribution. Let $C_t = \cov(Z_t)$ and $C'_t = \cov(Z'_t)$. Then we have $\rank(C'_t) = \rank(C_t)$.
\label{lemmaRankInvariance}
\end{lemma}

\begin{proof}
Note that $Z'_t = M^T Z_t$. We have
\begin{equation*}
\E[Z'_t] = \E[M^T Z_t] = M^T \E[Z_t]
\end{equation*}
and
\begin{equation*}
\E[Z'_t Z_t^{'T}] = \E[M^T Z_t Z_t^T M] = M^T \E[Z_t Z_t^T] M,
\end{equation*}
so clearly,
\begin{equation*}
C'_t = \cov[Z'_t] = M^T \cov[Z_t] M = C_t.
\end{equation*}
Since $M$ is invertible, it follows by elementary linear algebra that $\rank(C'_t) = \rank(C_t)$.
\end{proof}

Note that since $\rank(X) = d$, we may find an invertible matrix $M$ such that $X' = XM$ has the form
\[
X' 
=
\left[
\begin{array}{c}
I_{d} \\ \hline
H
\end{array}
\right],
\]
for some matrix \(H \in \real^{(m-d) \times d}\). By Lemma~\ref{lemmaRankInvariance}, it suffices to analyze the covariance matrix $C'_t$ of the sampled rows of $X'$. We claim that if $H\vect{1} \neq \vect{1}$, then $\rank(C'_t) = d$; if $H\vect{1} = \vect{1}$, then $\rank(C'_t) = d-1$.

We first consider the case when \(H\vect{1} \neq \vect{1}\).
Suppose $\rank(C'_t) < d$. Then for some \(\vect{v} \neq \vect{0}\), we have \(\vect{v}^{T} C'_t \vect{v} = 0\). We may write
\[
\vect{v}^{T} C'_{t} \vect{v} 
=
\var(\vect{v}^{T} Z'_{t}),
\]
where we use $Z'_t$ to denote the random vector corresponding to a randomly sampled row. Clearly, the latter expression is 0 if and only if $\vect{v}^T Z'_t,$ is almost surely constant. However, the possible values of $\vect{v}^T Z'_t$ correspond to the elements of \(X' \vect{v}\).
Since the upper block of \(X'\) is equal to \(I_{d}\) and \(\vect{v} \neq \vect{0}\), this constant must be nonzero because all rows are sampled with positive probability, by assumption. Without loss of generality, suppose
\begin{equation}
X' \vect{v} = \vect{1}.
\label{eqnVZConstant}
\end{equation}
Since we have \((X' \vect{v})_{i} = v_{i} = 1\) for \(i \leq d\), we must have \(\vect{v} = \vect{1}\).
But then $H \vect{1} = \vect{1}$, contradicting our assumption. Thus, we must have $\rank(C'_t) = d$.

In the case when $H \vect{1} = \vect{1}$, the above argument shows that $X' \vect{1} = \vect{1}$, so
\begin{equation*}
\vect{1}^T C'_{t} \vect{1} = \var(\vect{1}^T Z'_t) = 0.
\end{equation*}
In addition, any vector $\vect{v}$ satisfying $\vect{v}^T C'_t \vect{v} = 0$ must be a scalar multiple of $\vect{1}$. Hence, we conclude that the null space of $C'_t$ has dimension 1, so $\rank(C'_t) = d - 1$.


\section{Computational lemmas}
\label{AppComp}

In this Appendix, we derive a few useful lemmas regarding limiting distributions and expectations. We use \(\expect_{t}\) to denote the expectation with respect to the  conditional measure \(\prob_{\vbeta_{0}}(\cdot \mid \field_{t - 1})\), and write $\expect_\infty$ to denote the expectation with respect to the limiting conditional measure $\lim_{t \to \infty} \prob_{\vbeta_{0}, t} = \prob_{\vbeta_{0}, \infty} = \vect{\pi}_{\vbeta_0}$.
Recall that $Z_t$ denotes a random covariate vector corresponding to the edge chosen by $\prob_{\vbeta_{0}}(\cdot \mid \field_{t - 1})$.

\begin{proposition}
For every collection of positive integers $\{j_1, \dots, j_I\}$, we have the almost sure convergence
\[
\lim_{t \to \infty} \expect_{t}\left[\prod_{i = 1}^{I} Z_{t}(s_{i})^{j_{i}}\right]
=
\expect_{\infty}\left[\prod_{i = 1}^{I} Z_{\infty}(s_{i})^{j_{i}}\right].
\]
\label{lemmaLimitExpt}
\end{proposition}

\begin{proof}
Writing out the left-hand side, we have
\begin{align*}
& \begin{aligned}
\lim_{t \to \infty} \expect_{t}\left[\prod_{i = 1}^{I} Z_{t}(s_{i})^{j_{i}}\right]
&=
\lim_{t \to \infty}
\sum_{u = 1}^{m} \frac{b_{t}(u) \exp(\vect{x}_{u}^{T} \vbeta_{0})}{\sum_{v = 1}^{m} b_{t}(v) \exp(\vect{x}_{v}^{T} \vbeta_{0})} \prod_{i = 1}^{I} x_u(s_{i})^{j_{i}} \\ 
&=
\sum_{u = 1}^{m} \frac{b_{\infty}(u) \exp(\vect{x}_{u}^{T} \vbeta_{0})}{\sum_{v = 1}^{m} b_{\infty}(v) \exp(\vect{x}_{v}^{T} \vbeta_{0})} \prod_{i = 1}^{I} x_u(s_{i})^{j_{i}} \\ 
&=
\expect_{\infty}\left[\prod_{i = 1}^{I} Z_\infty(s_{i})^{j_{i}}\right],
\end{aligned}
\end{align*}
where the second equality follows from the almost-sure convergence guarantee of Lemma~\ref{propPiWeights}.
\end{proof}

One useful special case of Proposition~\ref{lemmaLimitExpt} is the convergence of conditional covariances
\begin{equation*}
\lim_{t \rightarrow \infty} \cov_{t}(Z_{t}(a), Z_{t}(b)) \to \cov_{\infty}(Z_{t}(a), Z_{t}(b)).
\end{equation*}

The second lemma provides an expression for computing derivatives of moments.
\begin{lemma}
For every collection of positive integers $\{j_1, \dots, j_I\}$, we have
\[
\frac{\partial }{\partial \sbeta(r)} \expect_{t}\left[\prod_{i = 1}^{I} Z_{t}(s_{i})^{j_{i}}\right]
=
\expect_{t} \left[Z_{t}(r) \prod_{i = 1}^{I} Z_{t}(s_{i})^{j_{i}}\right]
-
\expect_{t}\left[Z_{t}(r)\right]
\expect_{t}\left[\prod_{i = 1}^{I} Z_{t}(s_{i})^{j_{i}}\right].
\]
\label{lemmaDiffMoment}
\end{lemma}

\begin{proof}
We may compute the partial derivative as follows:
\begin{align*}
& \begin{aligned}
\frac{\partial }{\partial \sbeta(r)} \expect_{t}\left[\prod_{i = 1}^{I} Z_{t}(s_{i})^{j_{i}}\right]
&=
\frac{\partial }{\partial \sbeta(r)} 
\sum_{u = 1}^{m} \frac{b_{t}(u) \exp(\vect{x}_{u}^{T} \vbeta)}{\sum_{v = 1}^{m} b_{t}(v) \exp(\vect{x}_{v}^{T} \vbeta)} 
\prod_{i = 1}^{I} x_{u}(s_{i})^{j_{i}} \\ 
&=
\sum_{u = 1}^{m}
\frac{b_{t}(u) \exp(\vect{x}_{u}^{T} \vbeta)}{\sum_{v = 1}^{m} b_{t}(v) \exp(\vect{x}_{v}^{T} \vbeta)} 
x_{u}(r) \prod_{i = 1}^{I} x_{u}(s_{i})^{j_{i}}
 \\ &\qquad -
\left(\frac{\sum_{v = 1}^{m} b_{t}(v) \exp(\vect{x}_{v}^{T} \vbeta)x_{v}(r)}{\sum_{v = 1}^{m} b_{t}(v) \exp(\vect{x}_{v}^{T} \vbeta)}
\right) \\ 
&\qquad \times
\left(
\sum_{u = 1}^{m}
\frac{b_{t}(u) \exp(\vect{x}_{u}^{T} \vbeta)}{\sum_{v = 1}^{m} b_{t}(v) \exp(\vect{x}_{v}^{T} \vbeta)} 
\prod_{i = 1}^{I} x_{u}(s_{i})^{j_{i}}
\right) \\ 
&=
\expect_{t}\left[Z_{t}(r) \prod_{i = 1}^{I} Z_{t}(s_{i})^{j_{i}}\right] 
-
\expect_{t}\left[Z_{t}(r)\right]
\expect_{t}\left[\prod_{i = 1}^{I} Z_{t}(s_{i})^{j_{i}}\right].
\end{aligned}
\end{align*}
This proves the desired result.
\end{proof}

\begin{remark}
Lemma~\ref{lemmaDiffMoment} is helpful for computing derivatives and showing that they are uniformly bounded.
In particular, we have
\begin{equation}
\nabla \ell(\vbeta; \ordered_{k})
=
\sum_{t = 2}^{k} \vect{x}_{e_{t}} - \sum_{t = 2}^{k} \expect_{t}[Z_{t}] = \sum_{t=2}^k \left(Z_t - \E_t[Z_t]\right).
\label{eqnLogLikelihood}
\end{equation}
Going to second derivatives, we obtain
\begin{align*}
\frac{\partial^{2} \ell(\vbeta; \ordered_{k})}{\partial\sbeta(r) \partial \sbeta(s)}
&=
\sum_{t = 2}^{k} \Big(\expect_{t}[Z_{t}(r)] \expect_{t}[Z_{t}(s)] - \expect_{t}[Z_{t}(r) Z_{t}(s)]\Big).
\end{align*}
Iterating one step further, we have
\begin{align*}
& \begin{aligned}
\frac{\partial^{3} \ell(\vbeta; \ordered_{k})}{\partial \sbeta(q)\partial\sbeta(r) \partial \sbeta(s)}
&=
\sum_{t = 2}^{k}
\Big(- \expect_{t}[Z_{t}(q) Z_{t}(r) Z_{t}(s)]
-2\expect_{t}[Z_{t}(q)] \expect_{t}[Z_{t}(r)] \expect_{t}[Z_{t}(s)] \\
&\qquad \qquad
+ \expect_{t}[Z_{t}(q) Z_{t}(r)] \expect_{t}[Z_{t}(s)]
+ \expect_{t}[Z_{t}(q) Z_{t}(s)] \expect_{t}[Z_{t}(r)] \\ 
& \qquad \qquad
+ \expect_{t}[Z_{t}(r) Z_{t}(s)] \expect_{t}[Z_{t}(q)]\Big).
\end{aligned}
\end{align*}
We now see that the third derivatives are uniformly bounded in the following sense:
Suppose \(B\) is a uniform bound on the entries of \(Z_{t}\), which we know exists because all entries are drawn from a fixed matrix \(X\). Then the third derivatives of the log-likelihood are all bounded by \(6k B^{3}\). Similarly, we may argue that the first and second derivatives are uniformly bounded.
\label{remarkDerivative}
\end{remark}


\section{Consistency and asymptotic normality}
\label{appConsistencyNormality}

\subsection{Proof of Lemma~\ref{LemBall}}
\label{AppLemBall}

Our approach is to use a Taylor expansion around \(\vbeta_{0}\). Using the mean-value form of the remainder term~\citep{CouJoh89}, we have
\begin{align*}
& \begin{aligned}
\frac{1}{k} \ell(\vbeta) - \frac{1}{k} \ell(\vbeta_{0}) 
&=
\frac{1}{k} \sum_{a = 1}^{d} A_{a}(\ordered_{k}) (\sbeta(a) - \sbeta_{0}(a)) \\ 
&\qquad 
+ \frac{1}{2k} \sum_{a = 1}^{d} \sum_{b = 1}^{d} 
B_{ab}(\ordered_{k}) (\sbeta(a) - \sbeta_{0}(a)) (\sbeta(b) - \sbeta_{0}(b)) \\ 
&\qquad 
+ \sum_{a = 1}^{d} \sum_{b = 1}^{d} \sum_{c = 1}^{d} C_{abc}(\ordered_k)
(\sbeta(a) - \sbeta_{0}(a)) (\sbeta(b) - \sbeta_{0}(b)) (\sbeta(c) - \sbeta_{0}(c)) \\
& \defeq
S_{1} + S_{2} + S_{3},
   \end{aligned}
\end{align*}
where by Remark~\ref{remarkDerivative}, we have
\begin{align*}
A_{a}(\ordered_{k})
&= \frac{\partial}{\partial \beta(a)} \ell(\vbeta; \ordered_{k}) \bigr|_{\vbeta = \vbeta_{0}} =
\sum_{t = 2}^{k} Z_{t}(a) - \expect_{t}[Z_{t}(a)], \\
B_{ab}(\ordered_{k}) &=
\frac{\partial^{2}}{\partial\sbeta(a) \partial \sbeta(b)} \ell(\vbeta; \ordered_{k}) \bigr|_{\vbeta = \vbeta_{0}} 
= -\sum_{t = 2}^{k} \cov_{t}(Z_{t}(a), Z_{t}(b)),
\end{align*}
and $C_{abc}(\ordered_k)$ satisfies
\begin{align*}
|C_{abc}(\ordered_{k})| & \le \frac{1}{6} \max_{a,b,c} \left|
\frac{\partial^{3} \ell(\beta)}{\partial \sbeta(a) \partial \sbeta(b) \partial \sbeta(c)} \right|
%
\le \frac{1}{6} \cdot 6k \; \max_{i,j} |X(i,j)|^3.
\end{align*}

Now we turn to bounding the sums, starting with \(S_{1}\).
By Lemma~\ref{thmLLNMartingales}, we have
\[
\frac{1}{k} A_{a}(\ordered_{k}) \asto 0,
\]
since $A_a(\ordered_k)$ is a sum of bounded martingale increments. Hence, with probability tending to \(1\), 
we have \(|A_{a}(\ordered_{k}) / k| \leq r^{2}\).
As a result, we see that
\begin{align*}
& \begin{aligned}
|S_{1}|
&\leq 
\sum_{a = 1}^{d} \frac{1}{k} |A_{a}(\ordered_{k})| |\sbeta(a) - \sbeta_{0}(a)| 
\leq d r^{3}.
\end{aligned}
\end{align*}
To bound $S_2$, note that
\[
\frac{1}{k} B_{ab}(\ordered_{k})
=
-\frac{1}{k} \sum_{t = 2}^{k} 
\cov_{t}(Z_{t}(a), Z_{t}(b))
\asto 
-\cov_{\infty}(Z_{\infty}(a), Z_{\infty}(b))
=
-I_{\infty, ab}(\vbeta_0),
\]
by Proposition~\ref{lemmaLimitExpt} and Lemma~\ref{lemmaCesaros}, where
\begin{equation}
I_{\infty, ab}(\vbeta)
=
\expect_\infty
\left[
\frac{\partial}{\partial \sbeta(a)} \ell(\vbeta; \ordered_{k}) 
\cdot 
\frac{\partial}{\partial \sbeta(b)} \ell(\vbeta; \ordered_{k}) 
\right].
\label{eqnInformation}
\end{equation}
We now write
\begin{align*}
& \begin{aligned}
S_{2}
&=
\frac{1}{2} \sum_{a = 1}^{d} \sum_{b = 1}^{d} - I_{\infty, ab}(\vbeta_{0}) (\sbeta(a) - \sbeta_{0}(a)) (\sbeta(b) - \sbeta_{0}(b)) 
\\ &\qquad 
+
\frac{1}{2} \sum_{a = 1}^{d} \sum_{b = 1}^{d} \left(\frac{1}{k} B_{ab}(\ordered_{k}) - (-I_{\infty, ab}(\vbeta_{0}))\right)
(\sbeta(a) - \sbeta_{0}(a)) (\sbeta(b) - \sbeta_{0}(b)).
\end{aligned}
\end{align*}
Using an argument similar to that for \(S_{1}\), we can bound the second term by \(d^{2} r^{3}\) with probability going to \(1\) as \(k \rightarrow \infty\). Furthermore, we may rewrite the first term as 
\[
Q(\vbeta)
=
-\frac{1}{2}
(\vbeta - \vbeta_{0})^{T} I_{\infty}(\vbeta_{0}) (\vbeta - \vbeta_{0}),
\]
where \(I_{\infty}\) is the \(d \times d\) matrix of all \(I_{\infty, ab}\).
Since \(I_{\infty}\) is positive definite, the largest eigenvalue of \(-I_{\infty}\) is strictly less than \(0\).
Thus, for sufficiently small \(r\) and appropriate constants \(\alpha, \alpha' > 0\),
we have
\[
S_{2} 
\leq 
- \alpha' r^{2} + d^{2} r^{3}
\leq 
-\alpha r^{2},
\]
with probability tending to \(1\).
Finally, note that by the above discussion, we have
\[
|S_{3}| \leq C' d^{3} r^{3} 
= 
C r^{3},
\]
for some constant \(C\).

Putting this all together, we see that
\[
\frac{1}{k} \ell(\vbeta) - \frac{1}{k} \ell(\vbeta_{0})
\leq 
dr^{3} - \alpha r^{2} + Cr^{3}
=
r^{2}(-\alpha + r(d + C)),
\]
with probability going to \(1\) as \(k \to \infty\). For sufficiently small $r$, the right-hand side is less than 0, proving the desired result.


\subsection{Proof of Lemma~\ref{LemConverge}}
\label{AppLemConverge}

In order to derive the convergence of various components, we will employ Lemma~\ref{thmMultiMartingaleCLT} in Appendix~\ref{appOtherTheorems} with  \(K_t = -t^{-1/2} I_{d} \) and
\begin{equation*}
M_{t} 
= 
\begin{cases}
(\ell'_{1}(\vbeta_{0}; \ordered_{t}), \ldots, \ell'_{d}(\vbeta_{0}; \ordered_{t})), & t \intext \mathbb{N}, \\
M_{\lfloor t \rfloor}, & t \text{ not in } \mathbb{N}.
\end{cases}
\end{equation*}
Noting that
\begin{align*}
& \begin{aligned}
\Delta M_{is}
&=
M_{is} - M_{is-} 
=
\begin{cases}
0, & s \text{ not in } \mathbb{N}, \\ 
Z_{s}(i) - \expect_{s}[Z_{s}(i)], & s \intext \mathbb{N},
\end{cases}
\end{aligned}
\end{align*}
we can easily check that this yields a right-continuous martingale with limits from the left.

We now verify the conditions of Lemma~\ref{thmMultiMartingaleCLT}.
Proving (a) is clear since \(k^{-1/2} I_{d}  \to 0\).
For (b), note that we have 
\[
\overline{K}_{it} = \sum_{j = 1}^{d} |K_{ji, t}| = \frac{1}{t^{1/2}}.
\]
Since all the \(Z_{t}(i)\)'s are bounded by some constant \(B\), we see that 
\[
\overline{K}_{it} \expect\left[\sup_{s \leq t} |\Delta M_{is}| \right] 
\leq 
\frac{2B}{t^{1/2}}
\to
0,
\]
which proves (b).

For (c) and (d), we derive almost sure convergence. Consider the case when $t = k$ is an integer. We have
\begin{align*}
(Q_k)_{ab} & = \sum_{t=2}^k \left(\ell_a'(\vect{\beta}_0; \ordered_t) - \ell_a'(\vect{\beta}_0; \ordered_{t-1})\right) \left(\ell_b'(\vect{\beta}_0; \ordered_t) - \ell_b'(\vect{\beta}_0; \ordered_{t-1})\right) \\
& =  \sum_{t=2}^k \Delta M_{at} \Delta M_{bt} \\
& = \sum_{t=2}^k \left(Z_{t}(a) - \expect_{t}[Z_{t}(a)]\right) \left(Z_{t}(b) - \expect_{t}[Z_{t}(b)]\right).
\end{align*}
By Proposition~\ref{lemmaLimitExpt} and Lemma~\ref{lemmaCesaros}, we therefore have
\begin{equation*}
K_k Q_k K_k^T = \frac{1}{k} \sum_{t=2}^k \cov_t(Z_t, Z_t) \asto \cov_\infty(Z_\infty, Z_\infty) = I_\infty(\vect{\beta}_0).
\end{equation*}
It is straightforward to extend the convergence to non-integral values of $k$, from which we obtain (c).

For (d), equation~\eqref{eqnLogLikelihood} gives
\begin{align*}
H_k & = \E[M_k M_k^T] \\
& = \E\left[\left(\sum_{t = 2}^{k} Z_t - \sum_{t = 2}^{k} \expect_{t}[Z_{t}]\right)\left(\sum_{t = 2}^{k} Z_t - \sum_{t = 2}^{k} \expect_{t}[Z_{t}]\right)^T\right] \\
& = \E\left[\left(\sum_{t=2}^k \left(Z_t - \E_t[Z_t]\right)\right)\left(\sum_{t=2}^k \left(Z_t - \E_t[Z_t]\right)\right)^T\right],
\end{align*}
where the non-diagonal terms cancel because $M_k$ is a martingale. We also know that
\begin{align*}
Z_t & \asto Z_\infty, \\
\E_t[Z_t] & \asto z_\infty,
\end{align*}
almost surely, where $Z_\infty$ is the covariate vector of an edge chosen according to the distribution $\pi_{\vbeta_0}$, and $z_\infty$ is a constant. In particular, we have the almost sure convergence
\begin{equation*}
\left(Z_t - \E_t[Z_t]\right)\left(Z_t - \E_t[Z_t]\right)^T \asto \left(Z_\infty - z_\infty\right)\left(Z_\infty - z_\infty\right)^T.
\end{equation*}
Furthermore, the $Z_t$'s are uniformly bounded by $\max_{i,j}|X(i,j)|$. By the Dominated Convergence Theorem, we may conclude that almost surely,
\begin{equation*}
\E\left[\left(Z_t - \E_t[Z_t]\right)\left(Z_t - \E_t[Z_t]\right)^T\right] \asto \E\left[\left(Z_\infty - z_\infty\right)\left(Z_\infty - z_\infty\right)^T\right].
\end{equation*}
Hence, by Lemma~\ref{lemmaCesaros}, we have
\begin{equation*}
K_k H_k K_k^T 
= 
\frac{1}{k}  \E\left[\left(\sum_{t=2}^k \left(Z_t - \E_t[Z_t]\right)\right)\left(\sum_{t=2}^k \left(Z_t - \E_t[Z_t]\right)\right)^T\right] \asto \E\left[\left(Z_\infty - z_\infty\right)\left(Z_\infty - z_\infty\right)^T\right],
\end{equation*}
almost surely, where the last expression is positive definite because $\vect{\pi}_{\vect{\beta}_0}$ is strictly positive. This implies (d).

Hence, by Lemma~\ref{thmMultiMartingaleCLT}, we have the convergence in distribution
\[
-\frac{1}{k^{1/2}} \nabla \ell(\vbeta_{0}; \ordered_{k})
\lawto
\normal(\vect{0}, I_{\infty}(\vbeta_{0})),
\]
implying that the vector \(T_{k} = (T_{1, k}, \ldots, T_{d, k})\) converges to a 
\(\normal(\vect{0}, I_{\infty}(\vbeta_{0}))\) random variable.

Finally, we show that \(A_{k} \probto -I_{\infty}(\vbeta_{0})\) in probability.
The first term of \(A_{ab, k}\) is 
\[
\frac{1}{k} \ell''_{ab}(\vbeta_{0})
= 
-\frac{1}{k} \sum_{t = 2}^{k} \cov_{t}(Z_{t}(a), Z_{t}(b)) 
\asto
-I_{\infty}(\vbeta_{0}),
\]
where the convergence is almost sure.
Additionally, recalling the uniform bound $|\ell'''_{abc}(\vbeta)| \le 6 B^3 k$ from Remark~\ref{remarkDerivative}, the second term of \(A_{ab, k}\) satisfies
\begin{align*}
& \begin{aligned}
\left|\frac{1}{2k} \sum_{c = 1}^{d} \left(\hat{\sbeta}_{k}(c) - \sbeta_{0}(c)\right) \ell'''_{abc}(\vbeta^{*})\right|
&\leq 
3 B^3 \sum_{c = 1}^{d} |\hat{\sbeta}_{k}(c) - \sbeta_{0}(c)|
\probto 
0,
\end{aligned}
\end{align*}
in probability, using the consistency of \(\hat{\vbeta}_{k}\) established in Appendix~\ref{AppLemBall}. This proves the desired result.

\section{Proofs for Section~\ref{secWithoutOrder}}
\label{AppWithoutOrder}

\subsection{Proof of Theorem~\ref{thmEmpiricalWeights}}

From P\'{o}lya urn theory~\citep{athreya1968}, we know that
\[
\frac{\vect{c}_{t}}{\|\vect{c}_{t}\|_{1}}
\asto
\vect{\pi},
\]
almost surely, where the limiting distribution \(\vect{\pi}\) is the leading left eigenvector of \(W(\vbeta_{0})\).

We further use the fact that the limiting constants \(b_{\infty}(1), \ldots, b_{\infty}(m)\) are uniquely determined by the relation
\begin{equation}
\prob_{\vbeta_{0}, \infty}(e)
=
\frac{b_{\infty}(e) w_{\vbeta_{0}}(e)}{b_{\infty}(1) w_{\vbeta_{0}}(1) + \cdots + b_{\infty}(m) w_{\vbeta_{0}}(m)}.
\label{eqnPBeta0}
\end{equation}
This is stated in the following lemma:
\begin{lemma}
The constants \(b_{\infty}(1), \ldots, b_{\infty}(m)\) appearing in equation~\eqref{eqnPBeta0} are unique up to a scale factor.
\label{lemmaUniqueB}
\end{lemma}

\begin{proof}
We can apply Lemma~\ref{lemmaLinearUnique} in the following manner:
Let \(\vect{\pi}_{\vbeta_{0}}\) be the leading left eigenvector of \(W(\vbeta_{0})\), normalized to be a probability distribution.
Denote the \(b_{\infty}(i)\)'s as a vector by 
\(\vect{b}\), 
and let the weights \(w_{\vbeta_{0}}(i)\) be written in a vector as \(\vect{w}_{\vbeta_{0}}\).
By Lemma~\ref{propPiWeights}, we have
\begin{align*}
& \begin{aligned}
\prob_{\vbeta_{0}, \infty}(i)
&=
\vect{\pi}_{\vbeta_{0}}(i)
=
\frac{b_{\infty}(i) w_{\vbeta_{0}}(i)}{b_{\infty}(1) w_{\vbeta_{0}}(1) + \cdots + b_{\infty}(m) w_{\vbeta_{0}}(m)},
   \end{aligned}
\end{align*}
for each \(i\) and some \(b_{\infty}(i)\), almost surely.
In vectors, this is equivalent to
\begin{align*}
& \begin{aligned}
\diag(\vect{w}_{\vbeta_{0}}) \vect{b} 
&=
\vect{\pi}_{\vbeta_{0}} \vect{w}_{\vbeta_{0}}^{T} \vect{b}.
   \end{aligned}
\end{align*}
Letting \(M = \diag(\vect{w}_{\vbeta_{0}}) - \vect{\pi}_{\vbeta_{0}} \vect{w}_{\vbeta_{0}}^{T}\),
we need to show that \(M \vect{b} = 0\).
Since \(\vect{b}\) need not be unique, we can specify the scale by requiring \(b_{\infty}(m) = 1\).
By Lemma~\ref{lemmaLinearUnique} in Appendix~\ref{appOtherTheorems}, the resulting \(\vect{b}\) is then unique.
\end{proof}

By Lemma~\ref{propPiWeights},  the almost-sure limit
\begin{equation*}
\vect{w}_\infty 
\defeq 
\lim_{k \rightarrow \infty} \frac{\frac{\vect{c}_{k}}{\|\vect{c}_{k}\|_1} \odiv \vect{b}_{k}}{\left\|\frac{\vect{c}_{k}}{\|\vect{c}_{k}\|_1} \odiv \vect{b}_{k}\right\|_1} = \frac{\vect{\pi}_{\vbeta_0} \odiv \vect{b}_\infty}{\|\vect{\pi}_{\vbeta_0} \odiv \vect{b}_\infty\|_1}
\end{equation*}
exists. Clearly, we have $\|\vect{w}_\infty\|_1 = 1$ and
\begin{equation*}
(\vect{b}_\infty \circ \vect{w}_\infty) \propto \vect{\pi_{\vbeta_0}}.
\end{equation*}
Since $\|\vect{\pi}_{\vbeta_0}\|_1 = 1$, we must have 
$\vect{b}_\infty \circ \vect{w}_\infty / (\vect{b}_\infty^T \vect{w}_\infty) = \vect{\pi}_{\vbeta_0}$. Comparing with equation~\eqref{eqnPiWeights} and using Lemma~\ref{lemmaUniqueB}, we conclude that $\vect{w}_\infty = \vect{w}$.


\subsection{Proof of Theorem~\ref{theoremCyclic}}
\label{subsecWOOProofs}

With the cyclic graph structure and conjectured weights in hand,
all that remains is to perform some simple algebraic checks to prove this theorem.
Let \(\vect{b}_{k + 1}'' = \vect{b}_{k + 1} - \vect{b}_{2}\). Since $\vect{c}_{k+1}(i) = \vect{b}''_{k+1}(i+1)$, it is easy to see that \(\vect{\pi} = \vect{b}_{k + 1}'' \circ \vect{w} / \vect{b}_{k + 1}^{''T}\vect{w}\).
Now all that remains is to show \(\vect{\pi}\) is the left leading eigenvector of the weight matrix, defined by
\[
\tilde{W}
=
\left[\begin{array}{ccccc}
 & \frac{b_{k + 1}''(3)}{b_{k + 1}''(2)} & & & \\ 
 & & \frac{b''_{k + 1}(4)}{b''_{k + 1}(3)} & & \\
 & & & \ddots & \\
 & & & & \frac{b''_{k + 1}(1)}{b''_{k + 1}(n)} \\ 
 \frac{b''_{k + 1}(2)}{b''_{k + 1}(1)} & & & & 
\end{array}\right]
=
\left[\begin{array}{ccccc}
 & \frac{c_{k + 1}(2)}{c_{k + 1}(1)} & & & \\ 
 & & \frac{c_{k + 1}(3)}{c_{k + 1}(2)} & & \\
 & & & \ddots & \\
 & & & & \frac{c_{k + 1}(n)}{c_{k + 1}(n-1)} \\ 
 \frac{c_{k + 1}(1)}{c_{k + 1}(n)} & & & & 
\end{array}\right],
\]
where the rows and columns correspond to the ordering of edges $(1,2), (2,3), \dots, (n-1)$. First note that
\begin{equation*}
\tilde{\vect{\pi}} = \left(\vect{c}_{k+1}(1), \dots, \vect{c}_{k+1}(n)\right)^T
\end{equation*}
satisfies
\(\tilde{\vect{\pi}}^{T} = \tilde{\vect{\pi}}^{T} \tilde{W}\),
so \(\vect{\pi}\) is clearly a left eigenvector with eigenvalue \(1\).
To see that this is in fact the leading left eigenvector, note that the characteristic polynomial of \(\tilde{W}\) is
\begin{equation*}
p(\lambda) = \det\left(\lambda I - \tilde{W}\right) = \lambda^{n} - 1.
\end{equation*}
Hence, we conclude that \(1\) is the maximal eigenvalue.
This proves the theorem.


\section{Auxiliary results}
\label{appOtherTheorems}

In this section, we provide a few useful lemmas and theorems. We begin with a linear-algebraic lemma, which is useful for proving limiting results regarding P\'{o}lya urns.

\begin{lemma}
Suppose \(\vect{y} \in \reals^{m}\) has strictly positive entries.
Suppose \(\vect{\pi} \in \reals^{m}\) has nonnegative entries and \(\|\vect{\pi}\|_{1} = 1\).
Define the matrix \(M = \diag(\vect{y}) - \vect{\pi} \vect{y}^{T}\). If
\[
M \vect{x}
= 
\vect{0},
\]
then $\vect{x}$ is a scalar multiple of
\[
\vect{v}
=
\left(
\frac{\pi_{1}}{y_{1}}, \, \frac{\pi_{2}}{y_{2}}, \, \ldots, \frac{\pi_{m}}{y_{m}}
\right)^T.
\]
\label{lemmaLinearUnique}
\end{lemma}
\begin{proof}
Assume without loss of generality that $\pi_m > 0$. Observe that we may write $M$ as
\begin{align}
& \begin{aligned}
M =
\left[
\begin{array}{cccc}
(\pi_{2} + \cdots + \pi_{m}) y_{1}
& -\pi_{1}y_{2}
& \ldots
& - \pi_{1} y_{m} \\
-\pi_{2} y_{1}
& (\pi_{1} + \pi_{3} + \cdots + \pi_{m}) y_{2}
& & - \pi_{2} y_{m} \\ 
\vdots & & \ddots & \vdots  \\
-\pi_{m} y_{1} 
& -\pi_{m} y_{2} & \ldots & (\pi_{1} + \cdots + \pi_{m - 1}) y_{m}
\end{array}
\right].
\label{eqnM1}
   \end{aligned}
\end{align}
Clearly, $Mv = 0$, implying that $\rank(M) \le m-1$. We need to show that the inequality is actually an equality.

Consider a linear combination of the first $m-1$ columns of $M$, weighted by the coefficients $(c_1, \dots, c_{m-1})$. Suppose the linear combination is equal to zero. Examining the first and last component of the resulting vector, we obtain the equations
\begin{align*}
c_1 (1-\pi_1) y_1 - c_2 \pi_1 y_2 - \cdots - c_{m-1} \pi_1 y_{m-1} & = 0, \\
-c_1 \pi_m y_1 - c_2 \pi_m y_2 - \cdots - c_{m-1} \pi_m y_{m-1} & = 0.
\end{align*}
Note that if $\pi_1 = 0$, the first equation implies that $c_1 = 0$. If $\pi \neq 0$, we may divide the first equation by $\pi_1$, divide the second equation by $\pi_m$, and take the difference to obtain
\begin{equation*}
\left(c_1 + \frac{c_1(1-\pi_1)}{\pi_1}\right) y_1 = 0.
\end{equation*}
Recalling that $y_1 > 0$ by assumption, we may rearrange this last equation to conclude that $c_1 = 0$. A similar argument shows that $c_i = 0$ for all $1 \le i \le m-1$. Hence, we conclude that the first $m-1$ columns of $M$ are linearly independent, implying that $\rank(M) = m-1$, as wanted.
\end{proof}

Next, we state Ces\`{a}ro's lemma, which may be found in standard analysis texts such as \cite[Lemma 15.5]{carothers2000}.
\begin{lemma}[Ces\`{a}ro's lemma]
Let \(\{a_{k}\}\) be a sequence of real numbers converging to \(a\).
Then we also have the convergence of the average
\[
\frac{1}{n} \sum_{k = 1}^{n} a_{k} \to a.
\]
\label{lemmaCesaros}
\end{lemma}
Next, we state a helpful computational lemma:

\begin{lemma}[Lemma 6.5.2 from \citealp{lehmann2006}]
Let \((T_{1, k}, \ldots, T_{d, k})\) be a sequence of random vectors converging in distribution to 
\((T_{1}, \ldots, T_{d})\).
Suppose for each fixed \(i\) and \(j\), \(\{A_{ij, k}\}_{k=1}^\infty\) is a sequence of random variables converging in probability to constants \(a_{ij}\), and the resulting matrix \(A = [a_{ij}]\) is invertible.
Define \(B = A^{-1}\).
Then the solutions \((Y_{1, k}, \ldots, Y_{d, k})\) of the system of equations
\[
\sum_{j = 1}^{d}
A_{ij, k} Y_{j, k}
=
T_{i, k}
\]
converge in distribution to the solutions \((Y_{1}, \ldots, Y_{d})\) of the system of equations
\[
\sum_{j = 1}^{d}
a_{ij} Y_{j}
=
T_{i},
\]
which are given by
\[
Y_{i} = \sum_{j = 1}^{d} b_{ij} T_{k}.
\]
\label{lemmaSolveLinear}
\end{lemma}
We also require a multidimensional martingale central limit theorem:

\begin{lemma}[Martingale central limit theorem from~\cite{kuchler1999}]
Let \(M_{t} = (M_{t}(1), \ldots, M_{t}(d))\) be a \(d\)-dimensional square-integrable martingale with respect to a filtration \(\field_{t}\), and suppose the sample paths of \(M_{t}\) are right-continuous and have limits from the left.
Suppose there exists a family of non-random \(d \times d\) matrices \(\{K_{t}: t > 0\}\), with \(t \mapsto K_{t}\) continuous. Finally, suppose that as \(t \to \infty\), we have
\begin{itemize}
\item[(a)]
\(K_{t} \to 0\);
\item[(b)]
\(\overline{K}_{it} \expect\left[\sup_{s \leq t}  |\Delta M_{is}|\right] \to 0\) for \(i = 1, \ldots, d\),
where we define the terms \(\overline{K}_{it} = \sum_{j = 1}^{d} K_{ji, t}\) and \(\Delta M_{is} = M_{is} - M_{is-}\);
\item[(c)]
\(K_{t} Q_{t} K_{t}^{T} \probto \Xi\) in probability, where $Q_t$ is the quadratic variation matrix of $M_t$ with $(i,j)$ entry equal to
\begin{equation*}
(Q_t)_{ij} = \lim_{\|P\| \rightarrow 0} \sum_{k=1}^n (M_{t_k}(i) - M_{t_{k-1}}(i))(M_{t_k}(j) - M_{t_{k-1}}(j)),
\end{equation*}
with the limit is taken over finer and finer partitions $P$ of the interval $[0,t]$,
and $\Xi$ is a random positive semidefinite matrix; and
\item[(d)]
\(K_{t} H_{t} K_{t}^{T} \probto \Sigma\) in probability, where 
\(H_{t} \defeq \expect\left[M_{t} M_{t}^{T}\right]\) and \(\Sigma\) is positive definite.
\end{itemize}
Then we have the convergence in distribution
\[
K_{t} M_{t} \lawto Z \sim \normal_{d}(\vect{0}, \Xi).
\]
\label{thmMultiMartingaleCLT}
\end{lemma}
Finally, we state a law of large numbers for sums of bounded martingale differences. The proof is an easy consequence of the Azuma-Hoeffding inequality.

\begin{lemma}[Law of large numbers for bounded martingale differences]
Let \(\Delta_{1}, \ldots, \Delta_{k}\) be a martingale difference sequence with increments bounded by a constant \(B\).
Then
\[
\frac{1}{k} \sum_{t = 1}^{k} \Delta_{t} \asto 0.
\]
\label{thmLLNMartingales}
\end{lemma}

\begin{proof}
The proof is an application of the Azuma-Hoeffding inequality and the Borel-Cantelli lemma.
By the Azuma-Hoeffding inequality, we have
\[
\prob\left\{ \biggr|\sum_{t = 1}^{k} \Delta_{t} \biggr| > \epsilon k\right\}
\leq 
2 \exp
\left(
-\frac{\epsilon^{2} k^{2}}{2Bk}
\right)
=
2
\exp
\left(
-\frac{\epsilon^{2} k}{2B}
\right).
\]
Summing over all \(k\) yields a finite sum, so by the Borel-Cantelli lemma, 
our sum converges to \(0\), almost surely.
\end{proof}

Finally, we state a theorem from \cite{athreya1968} on the almost sure convergence of branching processes.
Let \(X_{t}\) be a multitype continuous time Markov branching process.
Let \(A\) be the inifinitestimal generator of the mean matrix semigroup \(\{M(t): t \geq 0\}\), where 
\[
M_{ij}(t) = \expect\left[X_{j}(t) | X_{r}(0) = \delta_{ri}, r = 1, \ldots, m\right],
\]
and the \(\delta_{ri}\)'s are Kronecker deltas.
We assume the process is positive regular; i.e., \(A\) is irreducible and nonsingular.
\begin{theorem}[Theorem~1 of \cite{athreya1968}]
If the first moments exist, we have
\[
\lim _{t \to \infty} X_{t} e^{-\lambda_{1} t} = X_{\infty} \vect{v},
\]
almost surely, where \(X_{\infty}\) is a nonnegative random variable, \(\lambda_{1}\) is the maximal eigenvalue of \(A\), and \(\vect{v}\) is the normalized leading left eigenvector of \(A\).
\label{thm1Athreya}
\end{theorem}

%
%

\section{Simulation results}
\label{appSimulationResults}
\makeatletter
\afterpage{\global\setlength\@fpsep{\textheight}}
\makeatother

In this Appendix, we provide tables of results from the simulations discussed in Section~\ref{secSims}.
\subsection{Simulations on a directed cycle}
\label{secSimCycle}

In this subsection, we present results for a directed cycle without loops.
The first set of results are in Tables~\ref{tableCycle5075RMSE}, \ref{tableCycle50CI}, \ref{tableCycle75CI}, and \ref{tableCycle5075time}.
For each set of parameters \((n, d, k)\), we conducted \(500\) simulations,
For each simulation, the edge covariates were independent and identically distributed samples from a \(\normal_{d}(0, (0.01)^{2}I_{d})\) distribution.
Most of the data follow the general trends outlined in Section~\ref{secSims}.

For the results in Table~\ref{tableCycle2RMSE}, simulations were similarly conducted on a directed cycle without loops, of size $n = 2$ and dimension $d=1$. One possibly surprising observation is that the error is higher for a given \(k\) when \(n = 2\) than when \(n = 50\) or \(n = 75\). This is likely due to the larger variation in edge weights, owing to having more edges when \(n\) is larger. Indeed, we examine the trimmed root mean squared error, where we ignore estimates that differ from the true \(\sbeta\) by more than \(10\) and the proportion of trimmed runs. It is likely that a few runs had very large differences in the two edge weights, leading to large variance in the realized process and therefore also the estimates.

\begin{table*}[p]
\centering
\ra{1.3}
\begin{tabular}{@{}llllllll@{}}\toprule
dimension & infection size  & 
\multicolumn{6}{c}{root-mean squared error by method, \(n = 50\) and \(n = 75\)} \\
\cmidrule{1-8}
&&
emp & mle & gw &
emp & mle & gw\\ \midrule
\(d = 1\) 
& \(k = 25\)    & 62.7 & 33.3 & 153.3 & 63.6 & 33.0 & 69.4 \\
& \(k = 50\)    & 47.5 & 22.1 & 30.02 & 51.0 & 20.7 & 36.0 \\
& \(k = 100\)   & 44.7 & 14.2 & 22.30 & 40.5 & 13.4 & 19.4 \\
& \(k = 250\)   & 34.1 & 7.87 & 13.36 & 36.9 & 8.14 & 16.4 \\
& \(k = 400\)   & 30.9 & 6.33 & 14.31 & 34.7 & 6.94 & 12.9 \\
& \(k = 500\)   & 29.9 & 5.53 & 13.57 & 32.4 & 5.68 & 11.9 \\
& \(k = 1,000\) & 28.3 & 4.06 & 9.911 & 29.3 & 4.24 & 9.13 \\
\midrule
\(d = 5\) 
& \(k = 25\)    & 1498.7 & 4528.4 & 1591.8 & 1549.7 & 5488.8 & 1494.9 \\
& \(k = 50\)    & 1056.8 & 1656.5 & 570.47 & 4498.8 & 1613.0 & 1442.8 \\
& \(k = 100\)   & 304.33 & 120.86 & 146.86 & 262.46 & 80.586 & 125.06 \\
& \(k = 250\)   & 123.33 & 36.652 & 45.479 & 121.55 & 34.490 & 63.077 \\
& \(k = 400\)   & 105.68 & 23.679 & 38.406 & 106.96 & 24.524 & 37.384 \\
& \(k = 500\)   & 97.159 & 20.682 & 35.049 & 93.737 & 20.799 & 35.575 \\
& \(k = 1,000\) & 83.533 & 12.757 & 25.754 & 83.140 & 13.212 & 25.527 \\
\midrule
\(d = 10\) 
& \(k = 25\)    & 6099.27 & 1356.9 & 6584.8 & 14039.1 & 990.757 & 4766.9 \\
& \(k = 50\)    & 10207.9 & 1418.7 & 6974.6 & 8688.13 & 1332.34 & 4778.0 \\
& \(k = 100\)   & 7784.28 & 2638.7 & 6531.1 & 34585.3 & 43045.9 & 7574.2 \\
& \(k = 250\)   & 3337.33 & 4358.7 & 2077.5 & 4838.75 & 2021.39 & 1899.7 \\
& \(k = 400\)   & 4276.97 & 1734.4 & 2214.3 & 10704.9 & 45498.8 & 3200.9 \\
& \(k = 500\)   & 3949.34 & 3448.2 & 1163.7 & 2299.84 & 725.327 & 1086.9 \\
& \(k = 1,000\) & 346.715 & 100.61 & 138.31 & 638.762 & 140.987 & 448.33 \\
\midrule
\(d = 20\) 
& \(k = 25\)    & 30994.25 & 1329.6 & 14947.0 & 19798.2 & 3032.4 & 15069.6 \\
& \(k = 50\)    & 18767.91 & 608.26 & 28866.3 & 9267.94 & 579.65 & 45271.8 \\
& \(k = 100\)   & 23701.71 & 793.28 & 20060.8 & 32806.8 & 493.67 & 17890.7 \\
& \(k = 250\)   & 17466.70 & 426.78 & 17549.8 & 26442.2 & 429.01 & 26400.5 \\
& \(k = 400\)   & 56383.87 & 440.49 & 63315.1 & 14160.4 & 427.40 & 20762.7 \\
& \(k = 500\)   & 242865.8 & 474.49 & 80426.8 & 34631.4 & 416.54 & 60524.4 \\
& \(k = 1,000\) & 15206.20 & 509.48 & 18019.2 & 11530.0 & 484.18 & 8915.05 \\
\bottomrule
\end{tabular}
\caption{Root mean squared error for \(\vbeta\) via the empirical distribution, the maximum likelihood estimator, and the general weights methods on a directed cycle, denoted by emp, mle, and gw, respectively.
The graph consists of \(n\) vertices, and the covariates are independent, identically distributed \(\normal_{d}(0, (0.01)^{2}I_{d})\) random variables.
For each simulation, \(k\) vertices are infected.
For each \(n\), \(d\), and \(k\), there are \(500\) simulations of the process.
In general, the maximum likelihood estimator seems to perform the best, followed by general weights and then the empirical estimators.
}
\label{tableCycle5075RMSE}
\end{table*}

\begin{table*}[p]
\centering
\ra{1.3}
\begin{tabular}{@{}llllll@{}}\toprule
dim.\ & inf.\ size  & 
\multicolumn{4}{c}{Confidence interval performance, \(n = 50\)} \\
\cmidrule{1-6}
&&
n.e.\ & cov.\ & avg.\ len. & nec.\ len. \\ \midrule
\(d = 1\) 
& \(k = 25\)    & 0 & 70 & \(5.79 \times 10^{2}\) & \(1.25 \times 10^{2}\)  \\
& \(k = 50\)    & 0 & 60 & \(3.36 \times 10^{1}\) & \(8.65 \times 10^{1}\)  \\
& \(k = 100\)   & 0 & 55 & \(2.00 \times 10^{1}\) & \(5.69 \times 10^{1}\)  \\
& \(k = 250\)   & 0 & 47 & \(9.84 \times 10^{0}\) & \(3.12 \times 10^{1}\)  \\
& \(k = 400\)   & 0 & 41 & \(6.76 \times 10^{0}\) & \(2.65 \times 10^{1}\)  \\
& \(k = 500\)   & 0 & 41 & \(5.85 \times 10^{0}\) & \(2.18 \times 10^{1}\)  \\
& \(k = 1,000\) & 0 & 31 & \(3.38 \times 10^{0}\) & \(1.64 \times 10^{1}\)  \\
\midrule
\(d = 5\) 
& \(k = 25\)    & 7 & 51 & \(1.04 \times 10^{7}\) & \(4.26 \times 10^{3}\)  \\
& \(k = 50\)    & 1 & 48 & \(1.28 \times 10^{2}\) & \(1.14 \times 10^{3}\)  \\
& \(k = 100\)   & 0 & 50 & \(3.59 \times 10^{1}\) & \(1.54 \times 10^{2}\)  \\
& \(k = 250\)   & 0 & 51 & \(1.54 \times 10^{1}\) & \(5.88 \times 10^{1}\)  \\
& \(k = 400\)   & 0 & 39 & \(1.06 \times 10^{1}\) & \(4.39 \times 10^{1}\)  \\
& \(k = 500\)   & 0 & 39 & \(8.50 \times 10^{0}\) & \(3.49 \times 10^{1}\)  \\
& \(k = 1,000\) & 0 & 35 & \(4.75 \times 10^{0}\) & \(2.16 \times 10^{1}\)  \\
\midrule
\(d = 10\) 
& \(k = 25\)    & 47 & 45 & \(3.84 \times 10^{9}\) & \(9.89 \times 10^{2}\)  \\
& \(k = 50\)    & 45 & 46 & \(1.97 \times 10^{9}\) & \(1.33 \times 10^{3}\)  \\
& \(k = 100\)   & 28 & 59 & \(3.84 \times 10^{8}\) & \(1.35 \times 10^{3}\)  \\
& \(k = 250\)   & 3  & 41 & \(7.05 \times 10^{6}\) & \(1.15 \times 10^{3}\)  \\
& \(k = 400\)   & 2  & 35 & \(2.41 \times 10^{5}\) & \(1.02 \times 10^{3}\) \\
& \(k = 500\)   & 1  & 31 & \(3.34 \times 10^{4}\) & \(7.95 \times 10^{2}\)  \\
& \(k = 1,000\) & 0  & 32 & \(1.07 \times 10^{1}\) & \(7.17 \times 10^{1}\)  \\
\midrule
\(d = 20\) 
& \(k = 25\)    & 47 & 53 & \(7.75 \times 10^{9}\) & \(4.47 \times 10^{2}\)  \\
& \(k = 50\)    & 51 & 49 & \(4.59 \times 10^{9}\) & \(4.34 \times 10^{2}\)  \\
& \(k = 100\)   & 49 & 51 & \(3.20 \times 10^{9}\) & \(4.68 \times 10^{2}\)  \\
& \(k = 250\)   & 49 & 51 & \(1.17 \times 10^{9}\) & \(3.69 \times 10^{2}\)  \\
& \(k = 400\)   & 46 & 54 & \(7.18 \times 10^{8}\) & \(4.09 \times 10^{2}\)  \\
& \(k = 500\)   & 51 & 49 & \(5.11 \times 10^{8}\) & \(3.60 \times 10^{2}\)  \\
& \(k = 1,000\) & 50 & 50 & \(2.22 \times 10^{8}\) & \(3.98 \times 10^{2}\)  \\
\bottomrule
\end{tabular}
\caption{Confidence interval performance on a directed cycle graph for the first coordinate of \(\vbeta\). The columns are n.e.\ for the percent of runs resulting in numerical errors, cov.\ for the percent of runs where the 95\% confidence interval contains \(\vbeta\), avg.\ len.\ for the average confidence interval length, and nec.\ len.\ for the $95^\text{th}$ quantile of the absolute distance of \(\hat{\vbeta}_{k}(1)\) from \(\vbeta(1)\).
The covariates are independent, identically distributed \(\normal_{d}(0, (0.01)^{2}I_{d})\) random variables.
For each \(d\) and \(k\), there are \(500\) simulations of the process.}
\label{tableCycle50CI}
\end{table*}

\begin{table*}[p]
\centering
\ra{1.3}
\begin{tabular}{@{}llllll@{}}\toprule
dim.\ & inf.\ size  & 
\multicolumn{4}{c}{Confidence interval performance, \(n = 75\)} \\
\cmidrule{1-6}
&&
n.e.\ & cov.\ & avg.\ len. & nec.\ len. \\ \midrule
\(d = 1\) 
& \(k = 25\)    & 0 & 70 & \(5.85 \times 10^{1}\) & \(1.40 \times 10^{2}\) \\
& \(k = 50\)    & 0 & 63 & \(3.43 \times 10^{1}\) & \(8.13 \times 10^{1}\) \\
& \(k = 100\)   & 0 & 56 & \(1.98 \times 10^{1}\) & \(5.29 \times 10^{1}\) \\
& \(k = 250\)   & 0 & 48 & \(9.79 \times 10^{0}\) & \(3.30 \times 10^{1}\) \\
& \(k = 400\)   & 0 & 39 & \(7.00 \times 10^{0}\) & \(2.79 \times 10^{1}\) \\
& \(k = 500\)   & 0 & 42 & \(5.77 \times 10^{0}\) & \(2.22 \times 10^{1}\) \\
& \(k = 1,000\) & 0 & 34 & \(3.38 \times 10^{0}\) & \(1.57 \times 10^{1}\) \\
\midrule
\(d = 5\) 
& \(k = 25\)     & 7 & 51 & \(4.53 \times 10^{8}\) & \(5.08 \times 10^{3}\) \\
& \(k = 50\)     & 2 & 50 & \(1.28 \times 10^{2}\) & \(1.23 \times 10^{3}\) \\
& \(k = 100\)    & 0 & 53 & \(3.66 \times 10^{1}\) & \(1.42 \times 10^{2}\) \\
& \(k = 250\)    & 0 & 45 & \(1.57 \times 10^{1}\) & \(6.63 \times 10^{1}\) \\
& \(k = 400\)    & 0 & 38 & \(1.06 \times 10^{1}\) & \(4.45 \times 10^{1}\) \\
& \(k = 500\)    & 0 & 37 & \(8.50 \times 10^{0}\) & \(3.95 \times 10^{1}\) \\
& \(k = 1,000\)  & 0 & 36 & \(4.95 \times 10^{0}\) & \(2.33 \times 10^{1}\) \\
\midrule
\(d = 10\) 
& \(k = 25\)    & 47 & 43 & \(3.53 \times 10^{9}\) & \(9.79 \times 10^{2}\) \\
& \(k = 50\)    & 45 & 48 & \(1.76 \times 10^{10}\) & \(1.27 \times 10^{3}\) \\
& \(k = 100\)   & 27 & 57 & \(5.12 \times 10^{8}\) & \(1.43 \times 10^{3}\) \\
& \(k = 250\)   & 4  & 44 & \(2.97 \times 10^{7}\) & \(1.52 \times 10^{3}\) \\
& \(k = 400\)   & 1  & 34 & \(1.05 \times 10^{5}\) & \(1.23 \times 10^{3}\)\\
& \(k = 500\)   & 0  & 35 & \(1.03 \times 10^{4}\) & \(3.97 \times 10^{2}\) \\
& \(k = 1,000\) & 0  & 34 & \(1.24 \times 10^{1}\) & \(7.79 \times 10^{1}\) \\
\midrule
\(d = 20\) 
& \(k = 25\)    & 45 & 54 & \(7.57 \times 10^{9}\) & \(5.40 \times 10^{2}\) \\
& \(k = 50\)    & 51 & 49 & \(4.85 \times 10^{9}\) & \(4.82 \times 10^{2}\) \\
& \(k = 100\)   & 55 & 45 & \(3.02 \times 10^{9}\) & \(4.57 \times 10^{2}\) \\
& \(k = 250\)   & 46 & 54 & \(1.09 \times 10^{9}\) & \(3.70 \times 10^{2}\) \\
& \(k = 400\)   & 47 & 53 & \(5.75 \times 10^{8}\) & \(3.90 \times 10^{2}\) \\
& \(k = 500\)   & 51 & 49 & \(4.53 \times 10^{8}\) & \(3.86 \times 10^{2}\) \\
& \(k = 1,000\) & 54 & 46 & \(2.19 \times 10^{8}\) & \(4.03 \times 10^{2}\) \\
\bottomrule
\end{tabular}
\caption{Confidence interval performance on a directed cycle graph for the first coordinate of \(\vbeta\). The columns are n.e.\ for the percent of runs resulting in numerical errors, cov.\ for the percent of runs where the 95\% confidence interval contains \(\vbeta\), avg.\ len.\ for the average confidence interval length, and nec.\ len.\ for the $95^\text{th}$ quantile of the absolute distance of \(\hat{\vbeta}_{k}(1)\) from \(\vbeta(1)\).
The covariates are independent, identically distributed \(\normal_{d}(0, (0.01)^{2}I_{d})\) random variables.
For each \(d\) and \(k\), there are \(500\) simulations of the process.}
\label{tableCycle75CI}
\end{table*}

\begin{table*}[p]
\centering
\ra{1.3}
\begin{tabular}{@{}llllllll@{}}\toprule
dimension & infection size  & 
\multicolumn{6}{c}{average time (s) by method, \(n = 50\) and \(n = 75\)} \\
\cmidrule{1-8}
&&
emp & mle & gw &
emp & mle & gw\\ \midrule
\(d = 1\) 
& \(k = 25\)    & 0.003 & 0.200 & 0.204 & 0.004 & 0.197 & 0.208 \\
& \(k = 50\)    & 0.006 & 0.491 & 0.460 & 0.006 & 0.484 & 0.467 \\
& \(k = 100\)   & 0.011 & 1.355 & 1.163 & 0.012 & 1.332 & 1.157 \\
& \(k = 250\)   & 0.026 & 4.795 & 4.060 & 0.028 & 4.832 & 4.170 \\
& \(k = 400\)   & 0.041 & 9.260 & 8.164 & 0.044 & 9.360 & 8.355 \\
& \(k = 500\)   & 0.052 & 12.66 & 11.19 & 0.054 & 12.41 & 11.06 \\
& \(k = 1,000\) & 0.103 & 31.19 & 26.83 & 0.112 & 32.90 & 30.26 \\
\midrule
\(d = 5\) 
& \(k = 25\)    & 0.003 & 0.261 & 0.191 & 0.004 & 0.258 & 0.195 \\ 
& \(k = 50\)    & 0.006 & 0.524 & 0.451 & 0.006 & 0.527 & 0.471 \\
& \(k = 100\)   & 0.011 & 1.321 & 1.153 & 0.012 & 1.294 & 1.175 \\
& \(k = 250\)   & 0.026 & 4.620 & 4.098 & 0.028 & 4.649 & 4.214 \\
& \(k = 400\)   & 0.042 & 9.159 & 8.370 & 0.044 & 9.164 & 8.421 \\
& \(k = 500\)   & 0.051 & 12.24 & 11.00 & 0.054 & 12.14 & 11.02 \\
& \(k = 1,000\) & 0.105 & 32.13 & 27.84 & 0.109 & 31.09 & 27.35 \\
\midrule
\(d = 10\) 
& \(k = 25\)    & 0.003 & 0.289 & 0.200 & 0.004 & 0.295 & 0.211 \\
& \(k = 50\)    & 0.006 & 0.760 & 0.455 & 0.007 & 0.764 & 0.475 \\
& \(k = 100\)   & 0.011 & 2.039 & 1.148 & 0.012 & 2.022 & 1.172 \\
& \(k = 250\)   & 0.026 & 6.351 & 4.061 & 0.028 & 6.658 & 4.231 \\
& \(k = 400\)   & 0.041 & 10.64 & 8.177 & 0.044 & 10.98 & 8.308 \\
& \(k = 500\)   & 0.051 & 13.84 & 11.16 & 0.054 & 13.15 & 11.16 \\
& \(k = 1,000\) & 0.105 & 32.36 & 27.64 & 0.109 & 32.72 & 27.91 \\
\midrule
\(d = 20\) 
& \(k = 25\)    & 0.003 & 0.294 & 0.200 & 0.004 & 0.294 & 0.205 \\
& \(k = 50\)    & 0.006 & 0.714 & 0.462 & 0.007 & 0.705 & 0.475 \\
& \(k = 100\)   & 0.011 & 1.872 & 1.145 & 0.012 & 1.949 & 1.213 \\
& \(k = 250\)   & 0.026 & 7.261 & 4.201 & 0.028 & 7.192 & 4.306 \\
& \(k = 400\)   & 0.042 & 14.41 & 8.357 & 0.044 & 14.21 & 8.332 \\
& \(k = 500\)   & 0.051 & 19.69 & 11.11 & 0.056 & 19.87 & 11.41 \\
& \(k = 1,000\) & 0.106 & 52.78 & 28.25 & 0.110 & 51.91 & 28.04 \\
\bottomrule
\end{tabular}
\caption{Average time in seconds for estimating \(\vbeta\) via projecting the empirical distribution, the maximum likelihood estimator, and the general weights methods on a directed cycle, denoted emp, mle, and gw, respectively.
The covariates are independent, identically distributed \(\normal_{d}(0, (0.01)^{2}I_{d})\) random variables. 
For each simulation, \(k\) vertices are infected.
For each \(n\), \(d\), and \(k\), there are \(500\) simulations of the process.}
\label{tableCycle5075time}
\end{table*}

\begin{table*}[p]
\centering
\ra{1.3}
\begin{tabular}{lc@{\hskip 1.5in}llllll@{}}\toprule
\multicolumn{6}{c}{root-mean squared error by method, \(n = 2\) and \(d = 1\)} 
 \\ \midrule
infection size  
&& emp & fp & mle & gw \\ \midrule
\(k = 25\)    && 6786 & 7117 & 4661 & 4661 \\
\(k = 50\)    && 1233 & 1134 & 710  & 710 \\
\(k = 100\)   && 624  & 599  & 468  & 468 \\
\(k = 250\)   && 1503 & 1776 & 250  & 1754 \\
\(k = 400\)   && 123  & 126  & 150  & 150 \\
\(k = 500\)   && 174  & 170  & 270  & 270 \\
\(k = 1,000\) && 204  & 205  & 267  & 267 \\ 
\toprule
\multicolumn{6}{c}{trimmed root-mean squared error by method, \(n = 2\) and \(d = 1\)} 
 \\ \midrule
infection size  
&& emp & fp & mle & gw \\ \midrule
\(k = 25\)    && 5.8 & 2.9 & 5.4 & 5.4 \\
\(k = 50\)    && 6.0 & 6.3 & 5.3 & 5.3 \\
\(k = 100\)   && 5.7 & 5.7 & 5.3 & 5.3 \\
\(k = 250\)   && 5.1 & 5.2 & 5.2 & 5.2 \\
\(k = 400\)   && 5.2 & 5.0 & 5.1 & 5.1 \\
\(k = 500\)   && 4.9 & 4.9 & 4.8 & 4.8 \\
\(k = 1,000\) && 4.3 & 4.4 & 4.3 & 4.3 \\ 
\toprule
\multicolumn{6}{c}{percent remaining after trimming, \(n = 2\) and \(d = 1\)} 
 \\ \midrule
infection size  
&& emp & fp & mle & gw \\ \midrule
\(k = 25\)    && 17 & 27 & 20 & 20 \\
\(k = 50\)    && 28 & 27 & 32 & 32 \\
\(k = 100\)   && 36 & 35 & 35 & 35 \\
\(k = 250\)   && 47 & 48 & 53 & 53 \\
\(k = 400\)   && 56 & 56 & 61 & 61 \\
\(k = 500\)   && 61 & 61 & 67 & 67 \\
\(k = 1,000\) && 71 & 71 & 74 & 74 \\
\bottomrule
\end{tabular}
\caption{Root mean squared error, trimmed root mean squared error, and percent of parameters that were kept during trimming. 
The trimming eliminated estimates that were more than \(10\) from the actual value of \(\beta\).
The estimators used were the empirical distribution, fixed point, maximum likelihood, and general weights, denoted by emp, fp, mle, and gw, respectively.
The fixed point method behaves similarly to the empirical distribution estimator.}
\label{tableCycle2RMSE}
\end{table*}

\subsection{Simulations on a directed cycle with loops}
\label{secSimCycleLoops}

In this subsection, we present results for a directed cycle with loops.
The results are in Tables~\ref{tableCycleLoops5075RMSE}, \ref{tableCycleLoops50CI},
\ref{tableCycleLoops75CI}, and \ref{tableCycleLoops5075time}.
For each set of parameters \((n, d, k)\), we conducted \(500\) simulations.
For each simulation, the first \(d\) coordinates of edge covariates for inter-vertex edges were independent, identically distributed samples from a \(\normal_{d}(0, (0.01)^{2}I_{d})\) distribution, and the last coordinate was \(0\).
For self-loops, all of the covariates were the \((d + 1)^\text{th}\) standard basis vector \(\vect{e}_{d + 1}\).
Compared to the data without self-loops, the error is much higher.
This is likely due to the presence of more parameters and the high chance of infections following self-loops.
Note that  the Ebola data contains about \(10\) times more infections with self-loops than without.

\begin{table*}
\centering
\ra{1.3}
\begin{tabular}{@{}llllllll@{}}\toprule
dimension & infection size  & 
\multicolumn{6}{c}{root-mean squared error by method, \(n = 50\) and \(n = 75\)} \\
\cmidrule{1-8}
&&
emp & mle & gw &
emp & mle & gw\\ \midrule
\(d = 1\) 
& \(k = 100\)   & 110.1 & 58.9 & 285.4 & 130.9 & 42.8 & 227.0 \\
& \(k = 250\)   & 84.18 & 26.2 & 152.8 & 70.16 & 19.4 & 118.1 \\
& \(k = 400\)   & 66.01 & 18.1 & 110.3 & 69.51 & 18.7 & 120.2 \\
& \(k = 500\)   & 64.61 & 16.1 & 77.76 & 64.81 & 14.6 & 96.31 \\
& \(k = 1,000\) & 53.63 & 10.1 & 67.49 & 59.29 & 11.1 & 86.65 \\
& \(k = 2,000\) & 51.63 & 7.26 & 59.20 & 47.95 & 6.49 & 68.86 \\
& \(k = 3,000\) & 43.76 & 5.75 & 53.97 & 45.90 & 6.11 & 58.61 \\
 \midrule
\(d = 2\) 
& \(k = 100\)   & 1915.0 & 2606.2 & 1738.9 & 696.0 & 1393.9 & 1291.4 \\
& \(k = 250\)   & 152.71 & 48.131 & 264.61 & 151.1 & 54.737 & 301.33 \\
& \(k = 400\)   & 131.07 & 48.161 & 179.14 & 122.0 & 39.128 & 275.00 \\
& \(k = 500\)   & 115.37 & 29.372 & 161.30 & 109.0 & 30.207 & 209.60 \\
& \(k = 1,000\) & 92.460 & 25.010 & 118.57 & 86.94 & 20.156 & 95.640 \\
& \(k = 2,000\) & 84.664 & 13.354 & 96.076 & 81.05 & 13.153 & 73.388 \\
& \(k = 3,000\) & 73.615 & 11.111 & 77.843 & 77.21 & 10.283 & 78.122 \\
 \midrule
\(d = 5\)                                              
& \(k = 100\)   & 3466.99 & 7093.8 & 5608.55 & 7703.3 & 257324.7 & 13490.1 \\
& \(k = 250\)   & 4259.41 & 5751.7 & 19764.1 & 3885.0 & 43619.85 & 4995.12 \\
& \(k = 400\)   & 6564.88 & 5735.5 & 2780.33 & 2277.3 & 6861.845 & 6519.96 \\
& \(k = 500\)   & 14521.5 & 1980.9 & 5271.67 & 3191.2 & 2347.777 & 2185.80 \\
& \(k = 1,000\) & 2344.93 & 3317.8 & 1257.38 & 1674.2 & 1583.915 & 1538.17 \\
& \(k = 2,000\) & 909.887 & 210.43 & 484.895 & 412.32 & 498.417 & 580.296 \\
& \(k = 3,000\) & 342.086 & 81.669 & 574.231 & 269.67 & 63.604 & 266.18 \\
\bottomrule
\end{tabular}
\caption{Root mean squared error for \(\vbeta\) via the empirical distribution, the maximum likelihood, and the general weights methods on a directed cycle with loops, denoted by emp, mle, and gw, respectively.
The graph is on \(n\) vertices, and the covariates between vertices are independent, identically distributed \(\normal_{d}(0, (0.01)^{2}I_{d})\) random variables, redrawn for each simulation. 
For loops, the covariates are just the \((d + 1)^\text{th}\) standard basis vector.
For each simulation, \(k\) vertices are infected.
For each \(n\), \(d\), and \(k\), there are \(500\) simulations of the process.
Compared to the cycle without loops, the error is much larger with loops.
}
\label{tableCycleLoops5075RMSE}
\end{table*}

\begin{table*}[p]
\centering
\ra{1.3}
\begin{tabular}{@{}llllll@{}}\toprule
dim.\ & inf.\ size  & 
\multicolumn{4}{c}{Confidence interval performance, \(n = 50\)} \\
\cmidrule{1-6}
&&
n.e.\ & cov.\ & avg.\ len. & nec.\ len. \\ \midrule
\(d = 1\)
& \(k = 100\)   & 0 & 26 & 23.7 & 196.4  \\
& \(k = 250\)   & 0 & 24 & 11.9 & 106.9  \\
& \(k = 400\)   & 0 & 23 & 7.58 & 77.53  \\
& \(k = 500\)   & 0 & 20 & 6.62 & 66.48  \\
& \(k = 1,000\) & 0 & 18 & 3.89 & 42.53  \\
& \(k = 2,000\) & 0 & 17 & 2.29 & 30.28  \\
& \(k = 3,000\) & 0 & 15 & 1.65 & 24.49  \\
 \midrule
\(d = 2\)
& \(k = 100\)   & 0 & 24 & 36.1 & 481.5  \\
& \(k = 250\)   & 0 & 24 & 14.1 & 133.4  \\
& \(k = 400\)   & 0 & 22 & 9.77 & 98.10  \\
& \(k = 500\)   & 0 & 18 & 8.18 & 82.76  \\
& \(k = 1,000\) & 0 & 13 & 4.70 & 63.60  \\
& \(k = 2,000\) & 0 & 12 & 2.66 & 39.07  \\
& \(k = 3,000\) & 0 & 11 & 1.96 & 28.82  \\
 \midrule
\(d = 5\)
& \(k = 100\)   & 13 & 46 & \(1.77 \times 10^{8}\) & 3929.8 
 \\
& \(k = 250\)   & 6  & 23 & \(1.52 \times 10^{4}\) & 4172.5 
 \\
& \(k = 400\)   & 4  & 20 & \(1.47 \times 10^{3}\) & 3031.2 
 \\
& \(k = 500\)   & 3  & 18 & \(1.49 \times 10^{2}\) & 2429.8 
 \\
& \(k = 1,000\) & 1  & 11 & \(2.55 \times 10^{1}\) & 572.71 
 \\
& \(k = 2,000\) & 0  & 11 & \(7.00 \times 10^{0}\) & 170.86 
 \\
& \(k = 3,000\) & 0  & 9  & \(4.85 \times 10^{0}\) & 111.26 
 \\ 
\bottomrule
\end{tabular}
\caption{
Confidence interval performance on a directed cycle graph with loops for the first coordinate of \(\vbeta\). The columns are n.e.\ for the percent of runs resulting in numerical errors, cov.\ for the percent of runs where the 95\% confidence interval contains \(\vbeta\), avg.\ len.\ for the average confidence interval length, and nec.\ len.\ for the $95^\text{th}$ quantile of the absolute distance of \(\hat{\vbeta}_{k}(1)\) from \(\vbeta(1)\).
The graph is on \(n\) vertices, and the covariates between vertices are independent, identically distributed \(\normal_{d}(0, (0.01)^{2}I_{d})\) random variables, redrawn for each simulation. 
For loops, the covariates are just the \((d + 1)^\text{th}\) standard basis vector.
For each simulation, \(k\) vertices are infected.
For each \(d\) and \(k\), there are \(500\) simulations of the process.}
\label{tableCycleLoops50CI}
\end{table*}

\begin{table*}[p]
\centering
\ra{1.3}
\begin{tabular}{@{}llllll@{}}\toprule
dim.\ & inf.\ size  & 
\multicolumn{4}{c}{Confidence interval performance, \(n = 75\)} \\
\cmidrule{1-6}
&&
n.e.\ & cov.\ & avg.\ len. & nec.\ len. \\ \midrule
\(d = 1\)
& \(k = 100\)   & 0 & 23 & 22.3 & 168.3 \\
& \(k = 250\)   & 0 & 25 & 10.4 & 75.80 \\
& \(k = 400\)   & 0 & 24 & 7.90 & 76.81 \\
& \(k = 500\)   & 0 & 23 & 6.65 & 57.45 \\
& \(k = 1,000\) & 0 & 16 & 3.78 & 49.30 \\
& \(k = 2,000\) & 0 & 15 & 2.23 & 26.66 \\
& \(k = 3,000\) & 0 & 13 & 1.62 & 24.18 \\
 \midrule
\(d = 2\)
& \(k = 100\)   & 0 & 30 & 31.0 & 381.5 \\
& \(k = 250\)   & 0 & 21 & 15.0 & 140.5 \\
& \(k = 400\)   & 0 & 21 & 9.48 & 96.38 \\
& \(k = 500\)   & 0 & 16 & 7.65 & 86.04 \\
& \(k = 1,000\) & 0 & 17 & 4.61 & 59.38 \\
& \(k = 2,000\) & 0 & 13 & 2.61 & 34.58 \\
& \(k = 3,000\) & 0 & 13 & 2.03 & 31.38 \\
 \midrule
\(d = 5\)
& \(k = 100\)   
& 15 & 52 & \(1.59 \times 10^{8}\) & 3858.4 \\
& \(k = 250\)   
& 5  & 25 & \(9.51 \times 10^{6}\) & 4039.8 \\
& \(k = 400\)   
& 3  & 18 & \(3.91 \times 10^{3}\) & 2375.5 \\
& \(k = 500\)   
& 2  & 18 & \(1.94 \times 10^{2}\) & 2609.3 \\
& \(k = 1,000\) 
& 1  & 13 & \(1.87 \times 10^{1}\) & 675.71  \\
& \(k = 2,000\)  
& 0  & 13 & \(6.68 \times 10^{0}\) & 160.47 \\
& \(k = 3,000\) 
& 0  & 10 & \(4.52 \times 10^{0}\) & 92.390 \\ 
\bottomrule
\end{tabular}
\caption{
Confidence interval performance on a directed cycle graph with loops for the first coordinate of \(\vbeta\). The columns are n.e.\ for the percent of runs resulting in numerical errors, cov.\ for the percent of runs where the 95\% confidence interval contains \(\vbeta\), avg.\ len.\ for the average confidence interval length, and nec.\ len.\ for the $95^\text{th}$ quantile of the absolute distance of \(\hat{\vbeta}_{k}(1)\) from \(\vbeta(1)\).
The graph is on \(n\) vertices, and the covariates between vertices are independent, identically distributed \(\normal_{d}(0, (0.01)^{2}I_{d})\) random variables, redrawn for each simulation. 
For loops, the covariates are just the \((d + 1)^\text{th}\) standard basis vector.
For each simulation, \(k\) vertices are infected.
For each \(d\) and \(k\), there are \(500\) simulations of the process.}
\label{tableCycleLoops75CI}
\end{table*}

\begin{table*}[p]
\centering
\ra{1.3}
\begin{tabular}{@{}llllllll@{}}\toprule
dimension & infection size  & 
\multicolumn{6}{c}{average time (s) by method, \(n = 50\) and \(n = 75\)} \\
\cmidrule{1-8}
&&
emp & mle & gw &
emp & mle & gw\\ \midrule
\(d = 1\) 
& \(k = 100\)   & 0.013 & 1.327 & 1.664 & 0.015 & 1.321 & 1.738 \\
& \(k = 250\)   & 0.030 & 4.535 & 5.335 & 0.034 & 4.542 & 5.337 \\
& \(k = 400\)   & 0.047 & 8.703 & 9.938 & 0.053 & 8.861 & 10.31 \\
& \(k = 500\)   & 0.058 & 11.75 & 12.94 & 0.067 & 11.84 & 14.06 \\
& \(k = 1,000\) & 0.122 & 31.38 & 34.89 & 0.133 & 29.91 & 33.66 \\
& \(k = 2,000\) & 0.232 & 74.62 & 86.37 & 0.262 & 74.08 & 90.81 \\
& \(k = 3,000\) & 0.363 & 131.1 & 166.4 & 0.391 & 124.3 & 160.9\\
 \midrule
\(d = 2\) 
& \(k = 100\)   & 0.013 & 1.374 & 1.622 & 0.016 & 1.410 & 1.747 \\
& \(k = 250\)   & 0.030 & 4.641 & 5.334 & 0.034 & 4.605 & 5.619 \\
& \(k = 400\)   & 0.047 & 8.767 & 9.980 & 0.052 & 8.658 & 9.939 \\
& \(k = 500\)   & 0.058 & 11.93 & 13.51 & 0.066 & 11.98 & 13.94 \\
& \(k = 1,000\) & 0.118 & 30.40 & 32.73 & 0.133 & 30.28 & 35.46 \\
& \(k = 2,000\) & 0.232 & 76.52 & 88.53 & 0.256 & 75.75 & 87.97 \\
& \(k = 3,000\) & 0.346 & 131.1 & 157.1 & 0.386 & 129.0 & 162.8 \\
 \midrule
\(d = 5\) 
& \(k = 100\)   & 0.013 & 2.761 & 1.757 & 0.015 & 2.602 & 1.752 \\
& \(k = 250\)   & 0.030 & 8.357 & 5.609 & 0.034 & 8.431 & 5.645 \\
& \(k = 400\)   & 0.047 & 14.32 & 9.819 & 0.053 & 14.15 & 10.50 \\
& \(k = 500\)   & 0.058 & 17.68 & 13.88 & 0.066 & 17.43 & 14.03 \\
& \(k = 1,000\) & 0.120 & 40.37 & 33.82 & 0.132 & 37.70 & 33.78 \\
& \(k = 2,000\) & 0.235 & 87.98 & 86.50 & 0.256 & 86.31 & 85.93 \\
& \(k = 3,000\) & 0.362 & 154.2 & 162.2 & 0.392 & 147.1 & 160.5
 \\
\bottomrule
\end{tabular}
\caption{Average time in seconds for estimating \(\vbeta\) via projecting the empirical distribution, the maximum likelihood estimator, and the general weights methods on a directed cycle, denoted by emp, mle, and gw, respectively.
The graph is on \(n\) vertices, and the covariates between vertices are independent, identically distributed \(\normal_{d}(0, (0.01)^{2}I_{d})\) random variables, redrawn for each simulation. 
For loops, the covariates are just the \((d + 1)^\text{th}\) standard basis vector.
For each simulation, \(k\) vertices are infected.
For each \(n\), \(d\), and \(k\), there are \(500\) simulations of the process.
Compared to the cycle without loops, the computation time seems much higher, particularly for the empirical distribution and general weights estimator.}
\label{tableCycleLoops5075time}
\end{table*}

\clearpage

\section{Ebola results}
\label{appEbolaResults}

In this Appendix, we provide the covariates used in our analysis of the spread of Ebola in West Africa described in Section~\ref{secApplications}.
Tables~\ref{tableEbolaCovariates1} and~\ref{tableEbolaCovariates2} contain descriptions of the covariates used.
All covariate information other than the data on shared borders comes from \cite{dudas2017}.
Finally, Table~\ref{tableEbolaResults} contains the results of the Ebola analysis.

\begin{table*}
\centering
\ra{1.3}
\begin{tabular}{llp{3in}}\toprule
type & covariate  & 
description \\
\midrule
geographic
& source temp.\ & Mean annual temperature of the source region, log-transformed and standardized. \\ 
& dest.\ temp.\ & Mean annual temperature of the destination region, log-transformed and standardized. \\ 
& source temp.\ seas.\ & Temperature seasonality index of the source region, log-transformed and standardized. \\ 
& dest.\ temp.\ seas.\ & Temperature seasonality index of the destination region, log-transformed and standardized. \\ 
& source prec.\ & Mean annual precipitation in the source region, log-transformed and standardized. \\ 
& dest.\ prec.\ & Mean annual precipitation in the destination region, log-transformed and standardized. \\ 
& source prec.\ seas.\ & Precipitation seasonality index of the source region, log-transformed and standardized. \\ 
& dest.\ prec.\ seas.\ & Precipitation seasonality index of the destination region, log-transformed and standardized. \\ 
\midrule
demographic
& gc distance & The great circle distance between population centroids, log-transformed and standardized. \\
& source pop.\ & The source population, log-transformed and standardized. \\ 
& dest.\ pop.\ & The destination population, log-transformed and standardized. \\ 
& source pop.\ density & The source population density, log-transformed and standardized. \\ 
& dest pop.\ density & The destination population density, log-transformed and standardized. \\ 
& source t.t.\ 100k & The estimated average travel time in the source region to the nearest settlement of 100,000 people, log-transformed and standardized. \\
& dest.\ t.t.\ 100k & The estimated average travel time in the destination region to the nearest settlement of 100,000 people, log-transformed and standardized. \\ 
& source econ.\ & Gridded economic output of the source, log-transformed and standardized. \\ 
& dest.\ econ.\ & Gridded economic output of the destination, log-transformed and standardized. \\
\bottomrule
\end{tabular}
\caption{Edge covariates used in the Ebola analysis. All indicators are \(0\) if the condition is not met and \(1\) if the condition is met.}
\label{tableEbolaCovariates1}
\end{table*}

\begin{table*}[p]
\centering
\ra{1.3}
\begin{tabular}{llp{3in}}\toprule
type & covariate  & 
description \\
\midrule
political
& dom.\ border & An indicator as to whether the source and destination regions share a border within the same country. \\
& int.\ border & An indicator as to whether the source and destination regions share a border but are in different countries. \\
& Guinea to Liberia & An indicator for the source being in Guinea and the destination being in Liberia. \\
& Guinea to Sierra Leone & An indicator for the source being in Guinea and the destination being in Sierra Leone. \\
& Liberia to Guinea & An indicator for the source being in Liberia and the destination being in Guinea. \\ 
& Liberia to Sierra Leone & An indicator for the source being in Liberia and the destination being in Sierra Leone. \\
& Sierra Leone to Guinea & An indicator for the source being in Sierra Leone and the destination being in Guinea. \\ 
& Sierra Leone to Liberia & An indicator for the source being in Sierra Leone and the destination being in Liberia. 
\\
\midrule
cultural 
& shared lang.\ dom.\ & An indicator for the source and destination being in the same country and sharing at least one of seventeen languages. \\
& shared lang.\ int.\ & An indicator for the source and destination being in different countries and sharing at least one of seventeen languages. \\
\bottomrule
\end{tabular}
\caption{Edge covariates used in the Ebola analysis. All indicators are \(0\) if the condition is not met and \(1\) if the condition is met.}
\label{tableEbolaCovariates2}
\end{table*}

\begin{table*}[p]
\centering
\ra{1.3}
\begin{tabular}{lSSl}\toprule
covariate & \mcL{coef.} & \mcL{std.\ err.\ } & abs.\ \(t\)-statistic \\
\midrule
gc distance 
& -0.594 & \(4.03 \times 10^{-4}\) & \(1.48 \times 10^{3}\) \\
dest.\ pop.\ 
& 0.749  & \(7.17 \times 10^{-4}\) & \(1.05 \times 10^{3}\) \\
source pop.\ 
& 0.948  & \(9.39 \times 10^{-4}\) & \(1.01 \times 10^{3}\) \\
int.\ border 
& 3.027  & \(3.12 \times 10^{-3}\) & \(9.75 \times 10^{2}\) \\
source t.t.\ 100k 
& 0.691  & \(7.69 \times 10^{-4}\) & \(8.99 \times 10^{2}\) \\
source prec.\ 
& 1.569  & \(2.20 \times 10^{-3}\) & \(7.14 \times 10^{2}\) \\
Sierra Leone to Guinea 
& -2.241 & \(3.15 \times 10^{-3}\) & \(7.12 \times 10^{2}\) \\
Guinea to Sierra Leone 
& -2.490 & \(3.59 \times 10^{-3}\) & \(6.96 \times 10^{2}\) \\
Liberia to Guinea 
& -2.418 & \(3.71 \times 10^{-3}\) & \(6.52 \times 10^{2}\) \\
Sierra Leone to Liberia 
& -3.117 & \(5.14 \times 10^{-3}\) & \(6.06 \times 10^{2}\) \\
Liberia to Sierra Leone 
& -3.866 & \(6.97 \times 10^{-3}\) & \(5.55 \times 10^{2}\) \\
Guinea to Liberia 
& -3.173 & \(5.78 \times 10^{-3}\) & \(5.50 \times 10^{2}\) \\
dest.\ prec.\ 
& 0.746  & \(1.40 \times 10^{-3}\) & \(5.33 \times 10^{2}\) \\
shared lang.\ dom.\ 
& 0.845  & \(1.83 \times 10^{-3}\) & \(4.61 \times 10^{2}\) \\
dest.\ t.t.\ 100k  
& 0.189  & \(7.11 \times 10^{-4}\) & \(2.65 \times 10^{2}\) \\
source pop.\ density 
& 0.312  & \(1.25 \times 10^{-3}\) & \(2.49 \times 10^{2}\) \\
source prec.\ seas.\ 
& -0.381 & \(1.54 \times 10^{-3}\) & \(2.48 \times 10^{2}\) \\
dest.\ temp.\  
& -0.154 & \(7.92 \times 10^{-4}\) & \(1.95 \times 10^{2}\) \\
source temp.\ 
& -0.210 & \(1.18 \times 10^{-4}\) & \(1.79 \times 10^{2}\) \\
source econ.\ 
& 0.085  & \(7.01 \times 10^{-4}\) & \(1.21 \times 10^{2}\) \\
dest.\ pop.\ density 
& 0.117  & \(1.00 \times 10^{-3}\) & \(1.16 \times 10^{2}\) \\
dest.\ prec.\ seas.\ 
& 0.117  & \(1.01 \times 10^{-3}\) & \(1.16 \times 10^{2}\) \\
shared lang.\ int.\ 
& 0.354  & \(3.30 \times 10^{-3}\) & \(1.07 \times 10^{2}\) \\
dest.\ econ.\ 
& 0.032  & \(5.59 \times 10^{-4}\) & \(5.64 \times 10^{1}\) \\
source temp.\ seas.\ 
& -0.093 & \(1.72 \times 10^{-3}\) & \(5.40 \times 10^{1}\) \\
dest.\ temp.\ seas.\ 
& 0.047  & \(1.13 \times 10^{-3}\) & \(4.14 \times 10^{1}\) \\
dom.\ border 
& 0.027  & \(9.74 \times 10^{-4}\) & \(2.77 \times 10^{1}\) \\
\bottomrule
\end{tabular}
\caption{Covariates, coefficients, standard errors, and absolute \(t\)-statistics for the maximum likelihood estimator analysis of the Ebola data, ordered by decreasing absolute \(t\)-statistic. 
The distance between regions and the populations of the regions are the most important, followed by effects related to international borders.}
\label{tableEbolaResults}
\end{table*}

\clearpage

\bibliography{networks}

\end{document}